\documentclass[reqno,twoside,11pt]{amsart}

\usepackage{amsmath, amsfonts, amssymb, esint, amsthm, nicefrac, bbm, dsfont, enumitem}
\usepackage{epsfig, graphicx}
\usepackage{hyperref, cite, comment}

\def\sideremark#1{\ifvmode\leavevmode\fi\vadjust{\vbox to0pt{\vss
 \hbox to 0pt{\hskip\hsize\hskip1em
\vbox{\hsize2cm\small\raggedright\pretolerance10000
 \noindent #1\hfill}\hss}\vbox to8pt{\vfil}\vss}}}

\setlist[itemize]{leftmargin=*}
\DeclareMathAlphabet{\mathpzc}{OT1}{pzc}{m}{it}
\DeclareMathOperator{\dist}{dist}

\newcommand{\dx}{\;\mathrm{d}x}
\newcommand{\dxp}{\;\mathrm{d}x'}
\newcommand{\R}{\mathbb{R}}

\newcommand{\G}{\mathcal{G}}
\def\endproof{\hspace*{\fill}\mbox{\ \rule{.1in}{.1in}}\medskip }

\renewcommand{\epsilon}{\varepsilon}

\numberwithin{equation}{section}

\theoremstyle{plain}
\newtheorem{theorem}{Theorem}[section]
\newtheorem{lemma}[theorem]{Lemma}
\newtheorem{corollary}[theorem]{Corollary}
\newtheorem{proposition}[theorem]{Proposition}

\theoremstyle{definition}

\begin{document}
\title[Dimension reduction for thin films prestrained by shallow curvature]
{Dimension reduction for thin films prestrained by shallow curvature} 
\author{Silvia Jimenez-Bolanos and Marta Lewicka}
\address{S. Jimenez-Bolanos: Colgate University,  Department of Mathematics,
13 Oak Drive, Hamilton, NY 13346, USA}  
\address{M. Lewicka: University of Pittsburgh, Department of
  Mathematics, 139 University Place, Pittsburgh, PA 15260, USA} 
\email{sjimenez@colgate.edu, lewicka@pitt.edu} 

\begin{abstract}
We are concerned with the dimension reduction analysis for thin
three-dimensional elastic films, prestrained via Riemannian metrics 
with weak curvatures. 

For the prestrain inducing the incompatible version of the F\"oppl-von K\'arm\'an equations,
we find the $\Gamma$-limits of the rescaled energies, identify the
optimal energy scaling laws, and display the equivalent conditions for
optimality in terms of both the prestrain components and the
curvatures of the related Riemannian metrics. When the stretching-inducing
prestrain carries no in-plane modes, we
discover similarities with the previously described shallow shell models.  

In higher prestrain regimes, we prove new energy upper
  bounds by constructing deformations as the Kirchhoff-Love
  extensions of the highly perturbative, H\"older-regular solutions to
  the  Monge-Ampere equation obtained by means of convex integration.
\end{abstract}

\date{\today}
\maketitle

\section{Introduction}

This paper is concerned with the dimension reduction analysis for
prestrained thin three-dimensional elastic films. We assume that the
prestrain corresponds to a family of Riemannian metrics 
with weak curvatures i.e. metrics deviating from the Euclidean metric by
the order of power of the film's thickness.
In various regimes of these scaling powers (separately for the
stretching and bending-inducing prestrain), we complete the dimension reduction in the full range of
parameters, as well as recover previous results in a unified manner.

There are essentially three new contributions in this paper:
\begin{enumerate}
\item[{(i)}]  For the prestrain of order at least as that inducing 
the incompatible version of the F\"oppl-von K\'arm\'an equations in \cite{lemapa1},
we: find the $\Gamma$-limits of the rescaled energies; identify the
optimal energy scaling laws; and display the equivalent conditions for
optimality in terms of both the prestrain components and the
curvatures of the related Riemannian metrics.
Similarly to the case of large prestrain \cite{LRR}, we observe that
one such condition is the non-vanishing of the lowest order
terms in the curvatures $R_{12,12}, R_{12, 13}, R_{13,23}$ along
the midplate. These results are valid  when $ \alpha\geq 4,\gamma\geq 2$ in (\ref{Ah}).

\item[{(ii)}] In the larger prestrain regime we propose new energy upper
  bounds, based on the construction of a sequence of deformations
  via the Kirchhoff-Love extension of the highly perturbative, H\"older-regular solutions to the
  Monge-Ampere equation obtained by convex integration.
These results are valid when $\alpha\in (0,4)$, $\gamma>0$ in (\ref{Ah}).

\item[{(iii)}] When the stretching-inducing prestrain is of order lower
  than that allowed in (i), but carries no in-plane modes, we still perform the full
  analysis as in (i) and discover similarities with both the theories
  in (i) and the shallow shell models of \cite{lemapa2}.
This corresponds to the case  $\alpha,\gamma\geq 2$ and $S_{2\times 2}\equiv 0$ in (\ref{Ah}).
\end{enumerate}

The remaining part of the introductory section will be devoted to the
description of our results, and their discussion including connection to previous literature.

\subsection{The set-up}

Thin prestrained films arise in science and technology in a variety of si\-tuations
and from a range of causes: inhomogeneous growth, plastic deformation,
swelling or shrinkage driven by solvent absorption. In all these
situations, the resulting shape is a consequence of the local and
heterogeneous incompatibility of strains that leads to local elastic
stresses. One approach towards understanding the coupling between residual stress and the
ultimate shape of the body relies on the model of ``non-Euclidean
elasticity''.  

The model postulates that an elastic three-dimensional film $\Omega$ seeks to realize a
configuration with a prescribed Riemannian metric
$\mathcal{G}$. Although $\G$ always has a Lipschitz isometric
immersion, one can show that any such immersion $u:\Omega\to \R^3$ necessarily
changes its orientation in any neighbourhood of a point where the
Riemann curvature of $\G$ is not zero. Excluding such nonphysical
deformations leads to the elastic energy $I(u)$ which measures how far a
given deformation $u$ is from being an orientation preserving
realization of $\G$. Equivalently, $I(u)$ quantifies the total
pointwise deviation of the deformation gradient $\nabla u$ from
$\G^{1/2}$, modulo orientation-preserving rotations. The infimum of $I$ in absence of any forces or
boundary conditions is then indeed strictly positive for a
non-Euclidean $\G$, pointing to existence of residual stress.

This approach borrows from the theory of plasticity, in as much as it uses a multiplicative decomposition
of the deformation gradient, and requires the notion of a reference
configuration $\Omega$ with respect to which all displacements are measured.
We assume that the elastic response derives from hyperelasticity, while the
inelastic deformation follows different laws depending on their
origin and it is encoded in the given prestrain $A=\sqrt{\G}$.
We point out that this description follows the
one-way coupling of growth to shape and ignores the feedback from
shape to growth. On the frontiers of the related experimental
modeling \cite{26, 25, 51, 15}, we mention the halftone gel
lithography method for polymeric materials that can swell by imbibing
fluids. By blocking the ability of portions of plate to swell or causing
them to swell inhomogeneously, it is possible to have 
the plate assume a variety of deformed shapes. Even more sophisticated techniques of biomimetic
4d printing allow for engineering of the 3d shape-morphing systems that mimic nastic plant motions
where organs such as tendrils, leaves and flowers respond to the environmental stimuli \cite{22}.

\smallskip

In this paper we consider a family of $(\Omega^h, u^h, \G^h,
A^h, I^h)_{h>0}$ as described above, but given in function of the film's thickness parameter
$h$. Following a large body of previous literature on dimension
reduction in nonlinear and non-Euclidean elasticity (see e.g. \cite{Lew_last}
and references therein), we are interested in predicting the scaling of $\inf
I^h$ as $h\to 0$ and analyzing the asymptotic behaviour of the minimizing deformations $u^h$,
from the curvatures of the prestrain. In analogy to the case of
``shallow shells'' considered in \cite{lemapa2}, here we treat the
``shallow prestrain'', i.e. we assume that the family of imposed tensors $A^h$
consists of perturbation of $Id$ of the order that is a power of the
thickness. We explicitly distinguish between the stretching-generating
leading order prestrain $S$ and the bending-related $B$, and define:
\begin{equation}\label{Ah}
\begin{split} 
A^h(x', x_3) = &\;  Id_3+ h^{\alpha/2} S(x') + h^{\gamma/2}x_3B(x')
\\ & \; \mbox{for all }\; x=(x', x_3)\in \Omega^h=\omega \times \big(-\frac{h}{2}, \frac{h}{2}\big),
\end{split}
\end{equation}
where $S,B:\bar\omega\to \R^{3\times 3}_{sym}$ 
and $\alpha,\gamma> 0$. The open, bounded set $\omega\subset\mathbb{R}^2$
with Lipschitz boundary is viewed as the midplate of the thin film
$\Omega^h$, on which we pose the energy of elastic deformations:
\begin{equation}\label{Ih}
I^h(u^h) = \frac{1}{h}\int_{\Omega^h} W\big((\nabla
u^h)(A^h)^{-1}\big)\dx\qquad \mbox{ for all }\; u^h\in W^{1,2}(\Omega^h,\R^3).
\end{equation}
Here, $W:\mathbb{R}^{3\times 3}\to\R$ is a Borel measurable density function, 
assumed to be $\mathcal{C}^2$ in a neighborhood
of $SO(3)$ and to satisfy, for every $F\in \mathbb{R}^{3 \times 3}$
and  $R\in SO(3)$:
\begin{equation}\label{Wpro}
\begin{split}
& W(R)=0, \qquad W(RF)=W(F), \qquad W(F)\geq c\dist^2(F, SO(3)).
\end{split}
\end{equation}

\subsection{Gamma-convergence: case $\mathbf{ \alpha\geq 4,\gamma\geq 2}$}

Our first result analyzes the case of bending and stretching components
in $A^h$ (i.e.  terms $h^{\alpha/2}S$,  $h^{\gamma/2}x_3B$) with scaling at least $h^2$.

\begin{theorem}\label{th_main1}
Let $\alpha\geq 4$, $\gamma\geq 2$. We have the following energy scaling and $\Gamma$-limit results.
\begin{enumerate}[leftmargin=8mm]
\item[(i)] When $\gamma=2$ then $\,\inf I^h\leq Ch^4$, and
  $h^{-4}I^h\overset{\Gamma}{\longrightarrow} I$ where: 
$$I(v,w) = \frac{1}{24}\int_\omega \mathcal{Q}_2\big(\nabla^2v +
B_{2\times 2}\big)\dxp + \frac{1}{2}\int_\omega\mathcal{Q}_2\Big(\mathrm{sym}\,\nabla w +
\frac{1}{2}(\nabla v)^{\otimes 2}  - \left\{\begin{array}{ll}
0 & \mbox{ for }\; \alpha>4\\ S_{2\times 2} & \mbox{ for }\; \alpha=4\end{array}\right.\Big)\dxp.$$

\item[(ii)] When $\gamma\in (2,\alpha-2]$ then $\,\inf I^h\leq
  Ch^{2+\gamma}$, and $h^{-(2+\gamma)}I^h\overset{\Gamma}{\longrightarrow} I$ where:
$$I(v,w) = \frac{1}{24}\int_\omega \mathcal{Q}_2\big(\nabla^2v +
B_{2\times 2}\big)\dxp + \frac{1}{2}\int_\omega\mathcal{Q}_2\Big(\mathrm{sym}\,\nabla w - \left\{\begin{array}{ll}
0 & \mbox{ for }\; \gamma<\alpha -2\\ S_{2\times 2} & \mbox{ for }\;
\gamma = \alpha-2\end{array}\right.\Big)\dxp.$$

\item[(iii)] When $\gamma >\alpha-2$ then $\,\inf I^h\leq
  Ch^{\alpha}$, and $h^{-\alpha}I^h\overset{\Gamma}{\longrightarrow} I$ where:
$$I(v,w) = \frac{1}{24}\int_\omega \mathcal{Q}_2\big(\nabla^2v\big)\dxp
+ \frac{1}{2}\int_\omega\mathcal{Q}_2\Big(\mathrm{sym}\,\nabla w + 
\left\{\begin{array}{ll}
0 & \mbox{ for }\; \alpha>4\\ \frac{1}{2}(\nabla v)^{\otimes 2} & \mbox{ for }\; \alpha=4\end{array}\right.\dxp
- S_{2\times 2} \Big)\dxp.$$ 
\end{enumerate}
All $\Gamma$-limit functionals $I(v,w)$ above are defined on the scalar out-of-plane
displacements $v\in W^{2,2}(\omega, \R)$ and the in-plane displacements $w\in W^{1,2}(\omega, \R^2)$.
The $\Gamma$-convergence statements are with respect to the following
compactness properties (valid in each corresponding scaling regime, with convergence up to a subsequence):
\begin{equation}\label{pomoc}
\begin{split}
&y^h(x', x_3) \doteq (\bar R^h)^T u^h(x', hx_3)-c^h\to x' \quad \mbox{in } \; W^{1,2}(\omega, \R^3) 
\\ & \qquad\qquad\qquad \qquad\qquad\qquad \qquad\qquad\qquad \mbox{for some }\; \bar
R^h\in SO(3), c^h\in \R^3, \\ & V^h(x')\doteq
h^{-\delta/2}\fint_{-1/2}^{1/2}y^h(x', x_3) -x'\;\mathrm{d}x_3 \to
(0,0,v) \quad \mbox{in } \; W^{1,2}(\omega, \R^3), \\ &
h^{-1}(V^h_1, V^h_2)\rightharpoonup w \quad \mbox{in } \; W^{1,2}(\omega, \R^{2}),
\end{split}
\end{equation}
where $\delta=2$ in case (i), $\delta = \gamma$ in case (ii), and $\delta=\alpha-2$ in case (iii).
\end{theorem}

The above theorem encompasses several cases studied before.
The case $\gamma=2$, $\alpha=4$ in (i) has been covered in
\cite{lemapa1}, and the case $\alpha=2\gamma$, $\gamma>2$ in (ii) was analyzed
in \cite{JZ}. We also refer to the paper \cite{BS}, where the authors considered prestrain of the type: $A^h(x_3)
= \big(Id_3 + h^{a-1} b\big(\frac{x_3}{h}\big)\big)^{-1}$ with
$a>2$. In our notation and for smooth tensor $b$, this is equivalent to having: $\alpha=2a-2>2$,
$\gamma=\alpha-2$, and $S=-b(0)$, $B=-b'(0)$ constant. When $a=3$ this
leads to a subcase of (i), while when $a>3$ to a subcase of (ii) in
which $I(v,w)$ can be always minimized to $0$. Thus, the optimal scaling of
$I^h$ in that case must have scaling exponent larger than $2a-2=\alpha$, studied in \cite{BS}.
We remark that the main contribution of \cite{BS} was, however, allowing
$W$ to depend on $\frac{x_3}{h}$ and $b:(-\frac{1}{2},
\frac{1}{2})\to\R^{3\times 3}_{sym}$ to have regularity $L^\infty$.

When $S=B\equiv 0$ in $\omega$ then all three cases in Theorem \ref{th_main1} reduce to the von
Karman and linear theories in classical nonlinear elasticity, derived
in \cite{FJM}. On the other hand, the limiting functional in case (ii) when $\gamma=\alpha-2$
seems to be new with respect to the previous literature. Of interest
is also the $\Gamma$-limit in (iii). When $\alpha=4$, its minimization amounts to
finding the displacement $v$ whose combined magnitude of the total induced curvature $\nabla^2v$
and the deviation of the Gaussian curvature $\det\nabla^2v$
from the given $-\mathrm{curl}^T\mathrm{curl}\,S_{2\times
  2}$, is the smallest.

In this line, we further remark that the in-plane displacement $w$ is always
slaved to $S,B$ and $v$, and as such can be omitted all together:

\begin{corollary}\label{coro1}
In the context of Theorem \ref{th_main1}, assume additionally that
$\omega$ is simply connected. Then each stretching term may be replaced
by the following squared distance from the space $\big\{\mathrm{sym}\nabla
w;~ w\in W^{1,2}(\omega,\R^2)\big\}$, where $\delta=2$ in case (i),
$\delta = \gamma$ in case (ii), and $\delta=\alpha-2$ in (iii):  
$$\frac{1}{2}\min_{w\in W^{1,2}(\omega, \R^2)}\int_\omega\mathcal{Q}_2\Big(\mathrm{sym}\,\nabla w +
\left\{\begin{array}{ll} 0 & \mbox{ for }\; \delta>2\\\frac{1}{2}
    (\nabla v)^{\otimes 2} & \mbox{ for }\; \delta=2\end{array}\right. - \left\{\begin{array}{ll} 
0 & \mbox{ for }\; \alpha>2+\delta \\ S_{2\times 2} & \mbox{ for }\;
\alpha=2+\delta \end{array}\right. \Big)\dxp.$$
Consequently, we have the equivalences below, where the congruency
symbol $a\simeq b$ means that $a\leq Cb$ and $b\leq Ca$ with a
constant $C$ depending only on $\omega$:
\begin{enumerate}
\item[(i)] If  $\gamma=2$, $\alpha>4$ then $h^{-4}I^h\overset{\Gamma}{\longrightarrow} \bar I$, where: 
$\bar I(v)  \simeq \|\nabla^2v +B_{2\times 2}\|^2_{L^2(\omega)}+\|\det\nabla^2v\|^2_{H^{-2}(\omega)}$.
If $\gamma=2$, $\alpha=4$ then:
$h^{-4}I^h\overset{\Gamma}{\longrightarrow} \bar I$, where
\begin{equation*}
\begin{split}
\bar I(v) \simeq \|\nabla^2v +B_{2\times
  2}\|^2_{L^2(\omega)}+\|\det\nabla^2v + \mathrm{curl}^T\mathrm{curl} S_{2\times 2}\|^2_{H^{-2}(\omega)}.
\end{split}
\end{equation*}

\item[(ii)] If $\gamma\in (2,\alpha-2]$ then
  $h^{-(2+\gamma)}I^h\overset{\Gamma}{\longrightarrow} \bar I$, with:
$\bar I(v) \simeq \|\nabla^2v +B_{2\times
  2}\|^2_{L^2(\omega)}+\|\mathrm{curl}^T\mathrm{curl} S_{2\times 2}\|^2_{H^{-2}(\omega)}.$
Alternatively, the same rescaled energies $\Gamma$-converge to the constant limit:
$$\bar{\bar I} \equiv \min \bar I \simeq \|\mathrm{curl}\,B_{2\times
  2}\|^2_{H^{-1}(\omega)}+\|\mathrm{curl}^T\mathrm{curl} S_{2\times 2}\|^2_{H^{-2}(\omega)}.$$

\item[(iii)] If $\alpha=4$, $\gamma >2$ then
  $h^{-4}I^h\overset{\Gamma}{\longrightarrow} \bar I$, where:
$\bar I(v) \simeq \|\nabla^2v
\|^2_{L^2(\omega)}+\|\det\nabla^2 v + \mathrm{curl}^T\mathrm{curl} S_{2\times 2}\|^2_{H^{-2}(\omega)}. $
If $\gamma >\alpha-2>2$ then
$h^{-\alpha}I^h\overset{\Gamma}{\longrightarrow} \bar{\bar I}$, where:
$\bar{\bar I} \equiv \min I \simeq \|\mathrm{curl}^T\mathrm{curl} S_{2\times 2}\|^2_{H^{-2}(\omega)}.$
\end{enumerate}
All $\Gamma$-limit functionals $\bar I(v)$ are defined on the scalar out-of-plane
displacements $v\in W^{2,2}(\omega, \R)$.
\end{corollary}



Next, we observe the bound on the infimum of the following limiting functional
corresponding to $\gamma=2, \alpha=4$, which also contains cases (i)
and (iii). The bound is consistent with the optimality conditions in
Theorem \ref{th_main10}.

\begin{proposition}\label{infbd}
Let $\omega\subset\R^2$ be open, bounded, simply connected. Denote
$\bar B=B_{2\times 2}$, $\bar S= S_{2\times 2}$, and for $v\in W^{2,2}(\omega, \R)$ define:
$$\bar I_0(v) = \|\nabla^2v + \bar B
\|^2_{L^2(\omega)}+\|\det\nabla^2 v + \mathrm{curl}^T\mathrm{curl}
\,\bar S\|^2_{H^{-2}(\omega)}. $$ 
Then, with some constants $c,C>0$ depending only on $\omega$ we have:
\begin{equation*}
\begin{split}
\inf \bar I_0\leq & \; C \,\|\mathrm{curl}\, \bar
B\|_{H^{-1}(\omega)}^2\Big(1+\|\bar B\|^2_{L^\infty(\omega)} +
\|\mathrm{curl}\, \bar B\|_{H^{-1}(\omega)}^2\Big) \\ & + C\,\|\det
\bar B + \mathrm{curl}^T\mathrm{curl}
\,\bar S\|^2_{H^{-2}(\omega)}, \\ 
\inf \bar I_0\geq & \; c\, \|\mathrm{curl}\, \bar B\|_{H^{-1}}^2
+ \frac{c}{1 \vee a } \|\det \bar B + \mathrm{curl}^T\mathrm{curl}
\, \bar S\|^2_{H^{-2}(\omega)},
\end{split}
\end{equation*}
where $a= \|\bar B\|^2_{L^\infty} + \|\det \bar B + \mathrm{curl}^T\mathrm{curl}
\, \bar S\|^2_{H^{-2}} + \|\mathrm{curl}\, \bar
B\|_{H^{-1}}^2 (1+\|\bar B\|^2_{L^\infty} + \|\mathrm{curl}\, \bar B\|_{H^{-1}}^2)$.
\end{proposition}

We finally remark that the following two cases, which are not included
in Theorem \ref{th_main1}: $\alpha=2\gamma$, $\gamma\in (0,2)$, and 
$\alpha\in (2,4)$,$\gamma=\alpha-2$ with $S$ and $B$ constant, have
been discussed in \cite{LOP} and \cite{BS}, respectively. They 
lead to the limiting theories with  Monge-Ampere constraints. 
In the present general setting, this case corresponds to taking
$\delta\leq \alpha/2$ and it will appear in \cite{J}.

\subsection{Energy scaling: case $\mathbf{0<\alpha<4}$ and $\mathbf{\gamma>0}$}
Our second result provides the energy bound when the prestrain in
$A^h$ is of order higher than $h^2$. Namely, we have:

\begin{theorem}\label{th_general_scaling} 
Assume that $\omega\subset\R^2$ is simply connected and has
$\mathcal{C}^{1,1}$-regular boundary. Let $\alpha\in (0,4)$ and
$\gamma>0$. Then the following holds:
\begin{enumerate}[leftmargin=8mm]
\item[(i)] If $\alpha\in [4/7,4)$ and $5\alpha/6+2/3>2+\gamma$, then
  $\inf I^h\leq Ch^{2+\gamma}$.

\item[(ii)] If $\alpha\in [4/7,4)$ and $5\alpha/6+2/3\leq 2+\gamma$, then
  $\inf I^h\leq Ch^{\delta}$ for every $\delta\in (0, 5\alpha/6 + 2/3)$.

\item[(iii)] If $\alpha\in (0,4/7)$, then $\inf I^h\leq Ch^{2\alpha}$.
\end{enumerate}
When $B=0$ then either  $\inf I^h\leq Ch^{\delta}$ for every
$\delta\in (0, 5\alpha/6 + 2/3)$ or $\inf I^h\leq Ch^{2\alpha}$,
depending on whether $\alpha\in [4/7, 4]$ or $\alpha\in (0,4/7)$.
 \end{theorem}

Figure \ref{fig:exponents} shows a diagram depicting various cases and
the corresponding scaling exponents. The indicated energy bounds are obtained
by constructing deformations $u^h$ through the Kirchhoff-Love
extension corresponding to the out-of-plane and the in-plane displacements
$(v,w)$, with regularity $\mathcal{C}^{1,\alpha}$ and satisfying
$\frac{1}{2}(\nabla v)^{\otimes 2} + \mbox{sym}\,\nabla w = S_{2\times
  2}$. Existence of such displacements is guaranteed 
by techniques of convex integration \cite{CS}, for all $\alpha<1/5$. This
threshold implies the particular energy scaling bounds in Theorem \ref{th_general_scaling}.
If we had $v\in W^{2,2}$ and $w\in W^{1,2}$ satisfying the same
equation, then $\inf I^h$ may be further decreased. Indeed, in
\cite{LOP} we showed that existence of $v\in W^{2,2}(\omega,\R)$ with
$\det\nabla^2 v=-\mbox{curl}\,\mbox{curl}\,S_{2\times 2}$ yields
$\inf I^h\leq Ch^{\alpha/2+2}$, for any $\alpha\in (2,4)$ and
$\gamma=\alpha/4$. Naturally, this bound is superior in any of the cases (i)-(iii).

\begin{figure}[htbp]
\centering
\includegraphics[scale=0.25]{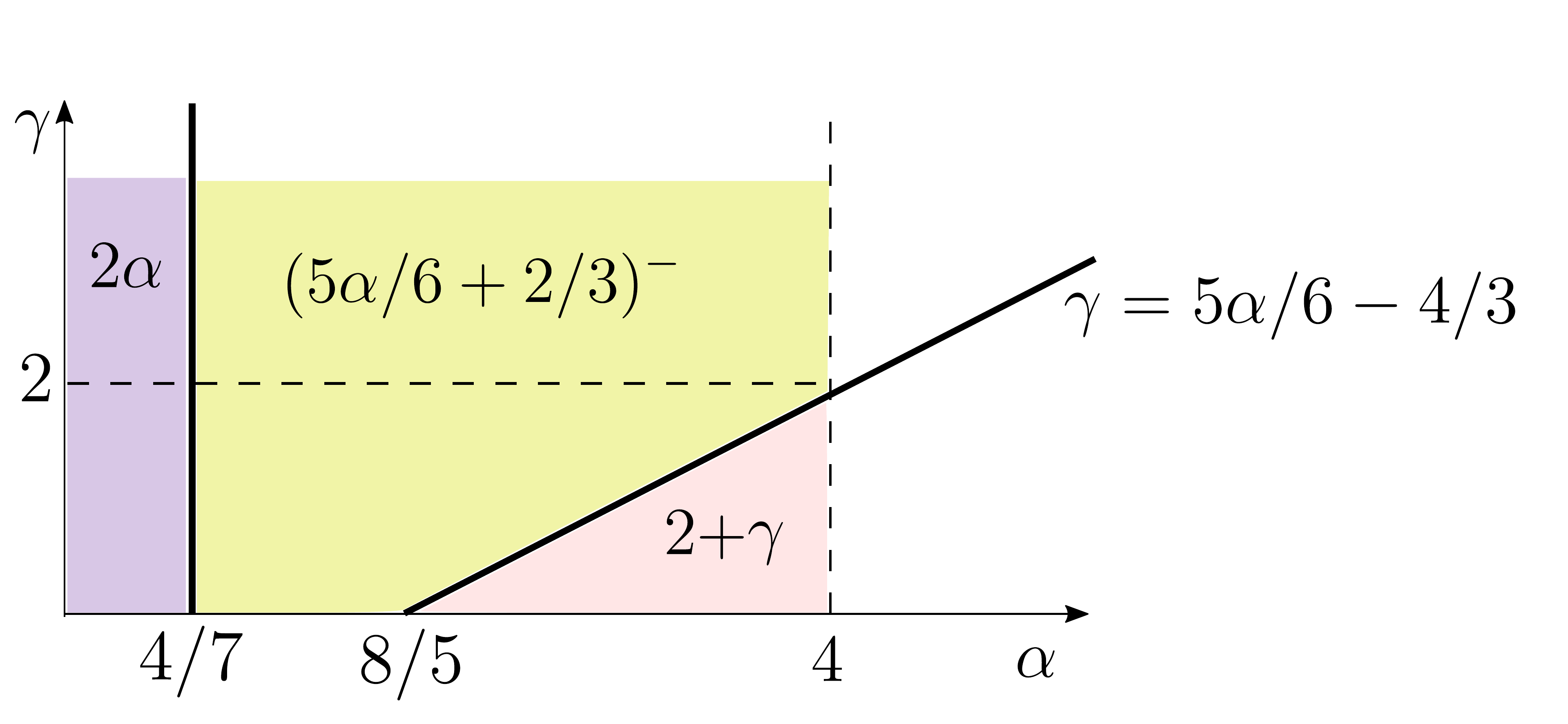}
\caption{{Bounding exponents of $\inf I^h$ in Theorem \ref{th_general_scaling} (i), (ii), (iii).}}
\label{fig:exponents}
\end{figure}

\subsection{Gamma-convergence: case $\mathbf{S_{2\times 2}\equiv 0}$
  and $\mathbf{\alpha,\gamma\geq 2}$}

Our next result concerns scaling of the bending prestrain
component $h^{\gamma/2}x_3B$ as in Theorem \ref{th_main1} (i.e. at least $h^2$), but
allowing for the stretching component $h^{\alpha/2}S$ in the
out-of-plane directions to be of order $h$, provided that the leading
order in-plane prestrain vanishes.

\begin{theorem}\label{th_main15} 
Let $S_{2\times 2}\equiv 0$ in $\omega$. For every $\alpha, \gamma\geq 2$ we have the
following energy scaling and $\Gamma$-convergence results.
\begin{enumerate}[leftmargin=8mm]
\item[(i)] When $\alpha=2$ then $\,\inf I^h\leq Ch^4$, and
  $h^{-4}I^h\overset{\Gamma}{\longrightarrow} I$ where: 
\begin{equation*}
\begin{split}
I(v,w) = & \; \frac{1}{24}\int_\omega \mathcal{Q}_2\big(\nabla^2v - 
2 \,\mathrm{sym} \nabla (S_{31}, S_{32}) +
\left\{\begin{array}{ll} 0 & \mbox{for } \; \gamma>2\\ B_{2\times
      2} &\mbox{for } \; \gamma =2\end{array}\right. \Big) \dxp 
\vspace{1mm} \\ & + \frac{1}{2}\int_\omega\mathcal{Q}_2\Big(\mathrm{sym}\,\nabla w +
\frac{1}{2}(\nabla v)^{\otimes 2}  - \frac{1}{2}(S_{31}, S_{32})^{\otimes 2} \Big)\dxp.
\end{split}
\end{equation*}
 
\item[(ii)] When $\gamma=2<\alpha$ then $\,\inf I^h\leq
  Ch^{4}$, and $h^{-4}I^h\overset{\Gamma}{\longrightarrow} I$ where:
$$I(v,w) = \frac{1}{24}\int_\omega \mathcal{Q}_2\big(\nabla^2v + B_{2\times 2}\big)\dxp +
\frac{1}{2}\int_\omega\mathcal{Q}_2\Big(\mathrm{sym}\,\nabla w +
\frac{1}{2} (\nabla v)^{\otimes 2}\Big)\dxp.$$
 
\item[(iii)] When $\alpha, \gamma >2$ then $\,\inf I^h\leq
  Ch^{2+\alpha\wedge \gamma}$, and $h^{-(2+\alpha\wedge\gamma)}I^h\overset{\Gamma}{\longrightarrow} I$ where:
\begin{equation*}
\begin{split}
& I(v,w) =  \frac{1}{2}\int_\omega\mathcal{Q}_2\big(\mathrm{sym}\,\nabla w \big)\dxp
\vspace{1mm} \\ & \quad + \frac{1}{24}\int_\omega \mathcal{Q}_2\Big(\nabla^2 v - 
\left\{\begin{array}{ll} 0 &\mbox{for }\; \gamma<\alpha\\ 2\,
    \mathrm{sym} \nabla (S_{31}, S_{32}) &\mbox{for }\; \alpha\leq\gamma\end{array}\right. +
\left\{\begin{array}{ll} 0 & \mbox{for } \; \alpha<\gamma\\ B_{2\times
      2} &\mbox{for } \; \gamma\leq\alpha\end{array}\right. \Big) \dxp.
\end{split}
\end{equation*}
\end{enumerate}
All $\Gamma$-limits $I(v,w)$ above are defined on the scalar out-of-plane
displacements $v\in W^{2,2}(\omega, \R)$ and the in-plane displacements
$w\in W^{1,2}(\omega, \R^2)$. 
The $\Gamma$-convergences are with respect to the compactness
statements in (\ref{pomoc}), which are valid in each corresponding
scaling regime, and with $\delta=2$ in case (i) and (ii), and $\delta=\alpha\wedge\gamma$ in case (iii).
\end{theorem}

We note that in cases (i) and (iii), the apparent stretching
components $S$ of $A^h$ contribute to the bending term and 
are mixed with the original bending components $B$. In fact,
the limiting functionals in those cases are of the same type as derived
in \cite[Theorem 5.2]{lemapa2} where we considered the prestrained
shallow shells. There, we assumed that the reference domains
$\Omega^h$ were configured around the mid-surfaces $\{x'+hv_0(x');~
x'\in\omega\}$ rather than the flat mid-plate
$\omega$. Taking $(S_{31}, S_{32})=\nabla v_0$ leads to the
same energy $I(v,w)$, in which both bending and stretching
are relative to the matching order tensors derived from $v_0$.
We point out that the presence of similar tensors also occurred in
\cite{LMP} in the context of shells with varying thickness.

We also check that results of Theorem \ref{th_main15} are stronger
than the general energy bounds in Theorem \ref{th_general_scaling},
which are however valid for any nonzero $S_{2\times 2}$.
Since $\alpha, \gamma\geq 2$, it follows that $\inf I^h\leq C
h^{(5\alpha/6+2/3)-}$ whenever $\alpha<4$. This leads to:
$h^{7/3-}$ in case (i), which is a bound indeed inferior to $h^4$. 
Likewise, in case (iii) we get: $h^{(5\alpha/6+2/3)-}\gg h^{2+\alpha \wedge\gamma}$.

Finally, observe that the in-plane displacement $w$ is slaved to $S,B$ and
$v$, and can be omitted by replacing the $\Gamma$-limits $I(v,w)$ by $\bar I(v)$:
\begin{corollary}\label{coro2}
In the context of Theorem \ref{th_main15}, assume additionally that
$\omega$ is simply connected. Then each stretching term may be replaced
by the following squared distance from the space $\big\{\mathrm{sym}\nabla
w;~ w\in W^{1,2}(\omega,\R^2)\big\}$, where $\delta=2$ in cases (i), (ii) and $\delta = \alpha\wedge \gamma$
in case (iii):
$$\frac{1}{2}\min_{w\in W^{1,2}(\omega, \R^2)}\int_\omega\mathcal{Q}_2\Big(\mathrm{sym}\,\nabla w +
\left\{\begin{array}{ll} 0 & \mbox{ for }\; \delta>2\\\frac{1}{2}
    (\nabla v)^{\otimes 2} & \mbox{ for }\; \delta=2\end{array}\right. - \left\{\begin{array}{ll} 
0 & \mbox{ for }\; \alpha>2\\ \frac{1}{2}(S_{31}, S_{32})^{\otimes 2} & \mbox{ for }\;
\alpha=2 \end{array}\right. \Big)\dxp.$$
Consequently, we have the equivalences of $\Gamma$-limits
$\bar I(v)$ defined on $v\in W^{2,2}(\omega, \R)$:
\begin{enumerate}
\item[(i)] If  $\alpha=2$, $\gamma>2$ then $h^{-4}I^h\overset{\Gamma}{\longrightarrow} \bar I$, where: 
$$\bar I(v)  \simeq \|\nabla^2v - 2\,\mathrm{sym}\nabla (S_{31}, S_{32})\|^2_{L^2(\omega)}+\|\det\nabla^2v-
\frac{1}{2}\mathrm{curl}^T\mathrm{curl}  (S_{31}, S_{32})^{\otimes 2}\|^2_{H^{-2}(\omega)}.$$
If $\alpha=\gamma=2$ then $h^{-4}I^h\overset{\Gamma}{\longrightarrow} \bar I$, where:
$$\bar I(v)  \simeq \|\nabla^2v - 2\,\mathrm{sym}\nabla (S_{31}, S_{32})
+B_{2\times 2}\|^2_{L^2(\omega)}+\|\det\nabla^2v-
\frac{1}{2}\mathrm{curl}^T\mathrm{curl}  (S_{31}, S_{32})^{\otimes 2}\|^2_{H^{-2}(\omega)}.$$

\item[(ii)] If $\gamma=2<\alpha$ then
  $h^{-4}I^h\overset{\Gamma}{\longrightarrow} \bar I$, where:
$\bar I(v) \simeq \|\nabla^2v +B_{2\times
  2}\|^2_{L^2(\omega)}+\|\det\nabla^2v\|^2_{H^{-2}(\omega)}.$

\item[(iii)] If $\alpha, \gamma >2$ then
  $h^{-(2+\alpha\wedge\gamma)}I^h\overset{\Gamma}{\longrightarrow} \bar{\bar I}$, where:
$$\bar{\bar I} \equiv \min I \simeq\left\{\begin{array}{ll} 
\|\mathrm{curl} B_{2\times 2}\|^2_{H^{-1}(\omega)} & \mbox{ for } \gamma<\alpha\vspace{1mm}\\
\|\nabla\mathrm{curl}  (S_{31}, S_{32})\|^2_{H^{-1}(\omega)} & \mbox{
  for }\gamma>\alpha \vspace{1mm}\\
\|\mathrm{curl} B_{2\times 2}- \nabla\mathrm{curl}  (S_{31}, S_{32})\|^2_{H^{-1}(\omega)} & \mbox{ for } \gamma=\alpha.\\
\end{array}\right.$$ 
\end{enumerate}
\end{corollary}
We remark that a bound on $\inf\bar I$ in cases (i) and (ii) above may
be deduced from Proposition \ref{infbd}, again consistent with Theorem \ref{th_main150}.



\subsection{Identification of the optimal scaling regimes}\label{identi_curv}

Our final set of results concerns the optimality of the energy
scalings implied by Theorems \ref{th_main1} and \ref{th_main15} and
their connection to curvature of the prestrains $A^h$. Namely, in the
setting of Theorem \ref{th_main1}, we get:

\begin{theorem}\label{th_main10}
Let $\alpha\geq 4$, $\gamma\geq 2$ and assume that $\omega$ is simply connected.
\begin{enumerate}[leftmargin=8mm]
\item[(i)] When $\gamma=2$ then $c h^4\leq \inf I^h\leq Ch^4$ with $c>0$, if and only if: 
$$\mathrm{curl}\, B_{2\times 2}\not\equiv 0, \quad \mbox{ or }\quad 
\det B_{2\times 2} + \left\{\begin{array}{ll}
0 & \mbox{ for }\; \alpha>4\\ \mathrm{curl}^T\mathrm{curl}\,S_{2\times
  2} & \mbox{ for }\; \alpha=4\end{array}\right. \not\equiv 0 \quad
\mbox{ in }\;\omega.$$
 
\item[(ii)] When $\gamma\in (2,\alpha-2]$ then $ch^{2+\gamma}\leq \inf I^h\leq
  Ch^{2+\gamma}$ with $c>0$, if and only if: 
$$\mathrm{curl}\, B_{2\times 2}\not\equiv 0, \quad \mbox{ or }\quad 
\gamma=\alpha -2 \;\mbox{ and }\; \mathrm{curl}^T\mathrm{curl}\,S_{2\times 2} \not\equiv 0 
\quad \mbox{ in }\;\omega.$$

\item[(iii)] When $\gamma >\alpha-2$ then $ch^{\alpha}\leq \inf I^h\leq
  Ch^{\alpha}$ with $c>0$, if and only if:
$$ \mathrm{curl}^T\mathrm{curl}\,S_{2\times 2} \not\equiv 0 \quad \mbox{ in }\; \omega.$$ 
\end{enumerate}
\end{theorem}

Consider now the following Riemann metrics on $\Omega^h$:
\begin{equation}\label{metric}
\G^h= (A^h)^TA^h = Id_3 + 2h^{\alpha/2}S + h^\alpha S^2 +
x_3\big(2h^{\gamma/2}B + 2 h^{(\alpha+\gamma)/2}\mbox{sym}(SB)\big) +
x_3^2h^\gamma B^2.
\end{equation}
We calculate the lowest order terms in the six Riemann
curvatures of $\G^h$ on $\omega$, at $x_3=0$:
\begin{equation}\label{curv_1}
\begin{split}
& R_{12, 12}\simeq - h^{\alpha/2} \mbox{curl}\,\mbox{curl}\,S_{2\times 2} -h^\gamma \det B_{2\times 2},\\
& R_{12, 13}\simeq - h^{\alpha/2} \partial_1\mbox{curl}(S_{13}, S_{23}) + h^{\gamma/2}(\mbox{curl}B_{2\times 2})_1, \\
& R_{12, 23}\simeq - h^{\alpha/2} \partial_2\mbox{curl}(S_{13}, S_{23}) + h^{\gamma/2}(\mbox{curl}B_{2\times 2})_2, \\
& \qquad R_{13, 13}\simeq - h^{\alpha/2} \partial^2_{11} S_{33} +
2h^{\gamma/2} \partial_1 B_{13},\\
& \qquad R_{13, 23}\simeq - h^{\alpha/2}\partial^2_{12} S_{33} +
h^{\gamma/2}\big(\partial_1 B_{23}+\partial_2 B_{13}\big),\\
& \qquad R_{23, 23}\simeq - h^{\alpha/2}\partial^2_{22} S_{33} + 2h^{\gamma/2}\partial_2 B_{23}.\\
\end{split}
\end{equation}
In all cases indicated in Theorem \ref{th_main10}, both the optimality conditions and the energy scaling
orders can be read from the first three curvatures above. These are: $R_{12,12}$
corresponding to stretching, and $R_{12,13}, R_{12,23}$ corresponding to
bending. Indeed, when $\gamma=2,\alpha=4$ we have $R_{12, 12}\simeq -
h^2(\mbox{curl}\,\mbox{curl}\,S_{2\times 2} +\det
B_{2\times 2})$ and $(R_{12, 13}, R_{12, 23})\simeq
h\,\mbox{curl}\,B_{2\times 2}$, which are the two tensors displayed in case
(i). When $\gamma=2,\alpha>4$ the only difference is
that $R_{12, 12}\simeq - h^2\det B_{2\times 2}$, and in both cases 
$\inf I^h$ is of order equal to the square of the stretching order,
i.e. $h^4$. When $\gamma>2, \alpha=\gamma+2$ we have $R_{12, 12}\simeq -
h^{\gamma/2+1}\mbox{curl}\,\mbox{curl}\,S_{2\times 2}$ and $(R_{12, 13},R_{12, 23})\simeq
h^{\gamma/2}\,\mbox{curl}\,B_{2\times 2}$, in agreement with case (ii). For $\gamma>2,\alpha<\gamma+2$ the
stretching-related curvature $R_{12, 12}$ has order $h^{\gamma\wedge \frac{\alpha}{2}}$
which is strictly less than the bending-compatible order
$h^{1+\gamma/2}$, so this term becomes irrelevant. The energy order
in both cases equals the square of the compatible stretching order,
i.e. $h^{\gamma+2}$. Finally, when $\gamma>\alpha-2$ we have: 
$R_{12, 12}\simeq - h^{\alpha/2} \mbox{curl}\,\mbox{curl}\,S_{2\times
  2}$, while $(R_{12, 13},R_{12, 23})$ has order $h^{\frac{\gamma}{2}\wedge \frac{\alpha}{2}}$
which is strictly less than the stretching-compatible 
$h^{\alpha/2 -1}$. Thus, the bending tensor is discarded on the basis
of not contributing towards the energy, which is in agreement with
case (iii).  Similarly, the energy
order equals the square of the stretching order i.e. $h^{\alpha}$.

We note that the above observations are precisely the
``small-curvature'' regime counterparts of the findings in
\cite{LRR}. There, the authors considered a constant prestrain
$A=A(x', x_3)$ independent of $h$. They proved that the optimal energy
scaling, which is $h^2$, is superseded (i.e. $h^{-2}\inf
I^h\to 0$ as $h\to 0$) if and only if $ R_{12, 12}=R_{12, 13}=R_{12, 23}\equiv 0$ on
$\omega$. The next viable scaling is then $h^4$, corresponding to the
three remaining curvatures $ R_{13, 13}, R_{13, 23}, R_{23, 23}$. 
Further, the vanishing of these implies the energy scaling order $h^6$,
and in general one has an infinite hierarchy of the limiting theories
each valid in an appropriate regime of vanishing of components of
$R_{\cdot\cdot, \cdot\cdot}$ together with their covariant derivatives \cite{Lew_last}.

We conjecture that the next viable energy scalings after
those proved in Theorem \ref{th_main10} are: $h^6$ in case (i),
$h^{(4+\gamma)\wedge 2\gamma}$ in case (ii) for $\gamma=\alpha-2$,
$h^{(4+\gamma)\wedge 2\gamma\wedge \alpha}$ in case (ii) for $\gamma<\alpha-2$, 
and $h^{(2+\gamma)\wedge (2+\alpha)}$ in case (iii). By analogy, vanishing of
the lowest order terms of $R_{13, 13}, R_{13, 23}, R_{23, 23}$ given in (\ref{curv_1})
should then be responsible for even higher energy scalings. 

\medskip

We begin the discussion of optimality in the setting of Theorem
\ref{th_main15} by stating:

\begin{theorem}\label{th_main150} 
Let $S_{2\times 2}\equiv 0$ in $\omega$ and let $\alpha, \gamma\geq 2$. 
Assume that $\omega$ is simply connected.
\begin{enumerate}[leftmargin=8mm]
\item[(i)] When $\alpha=2$ then $ch^4\leq \inf I^h\leq Ch^4$  with $c>0$, if and only if: 
\begin{equation*}
\begin{split}
&\mbox{when } \gamma>2: \quad \nabla' \, \mathrm{curl}  (S_{31},
S_{32})\not\equiv 0 \quad \mbox{ or }\quad \det\nabla' ( S_{31}, S_{32}) \not\equiv 0 \quad \mbox{ in }\;\omega, \\ 
&\mbox{when } \gamma=2: \quad \mathrm{curl}\, B_{2\times 2} -\nabla' \, \mathrm{curl}  (S_{31},
S_{32})\not\equiv 0 \quad \\ & \qquad\qquad \qquad \; \mbox{ or }
\det\big(B_{2\times 2} - 2\,\mathrm{sym}\nabla' (S_{31}, S_{32}) 
\big) + \frac{1}{2}\mathrm{curl}^T\mathrm{curl}(S_{31}, S_{32})^{\otimes 2}\not\equiv 0.
\end{split}
\end{equation*}

\noindent Equivalently, the last condition above may be rewritten as:
$$ 3 \det\nabla'(S_{31}, S_{32})
- \big\langle\nabla'\,\mathrm{curl}(S_{31}, S_{32}), (S_{31},
S_{32})^\perp\big\rangle- 2\big\langle \mathrm{cof}\, B_{2\times 2}:
\nabla' (S_{31}, S_{32})\big\rangle + \det B_{2\times 2} \not\equiv 0.$$

\item[(ii)] When $\gamma=2<\alpha$ then $ch^4\leq \inf I^h\leq Ch^{4}$ with $c>0$, if and only if: 
$$\mathrm{curl}\, B_{2\times 2}\not\equiv 0 \quad \mbox{ or }\quad 
\det B_{2\times 2}\not\equiv 0  \quad \mbox{ in }\;\omega.$$
 
\item[(iii)] When $\alpha, \gamma >2$ then $ch^{2+\alpha\wedge \gamma}\leq \inf I^h\leq
  Ch^{2+\alpha\wedge \gamma}$ with $c>0$, if and only if:
\begin{equation*}
\begin{split}
\mathrm{curl}\, B_{2\times 2}\not\equiv 0 \;\; \mbox{ and } \;\; \gamma<\alpha,
\quad & \mbox{ or }\quad \nabla' \,\mathrm{curl}(S_{31}, S_{32})\not\equiv 0 \;\; \mbox{ and } \;\; \gamma>\alpha,
\\ & \mbox{ or } \quad \mathrm{curl}\, B_{2\times 2} - \nabla' \,\mathrm{curl}(S_{31}, S_{32})
\not\equiv 0  \;\;\mbox{ and } \;\; \gamma=\alpha.
\end{split}
\end{equation*}
\end{enumerate}
\end{theorem}
Firstly, observe that (ii) above coincides with the statement of
Theorem \ref{th_main10} (i) when $S_{2\times 2}\equiv 0$, while (iii)
gives more precise information than Theorem \ref{th_main10} (ii) and (iii).
As before,  we may compute the lowest order terms in the Riemann
curvatures of the metric $\G^h$ in (\ref{metric}), at $x_3=0$:
\begin{equation}\label{curv_2}
\begin{split}
& R_{12, 12} \simeq h^\alpha \Big(-3 \det\nabla'(S_{31}, S_{32}) 
+ S_{13}\partial_2\mbox{curl}(S_{31}, S_{32}) - S_{32}\partial_1\mbox{curl}(S_{31}, S_{32}) \Big)
\\ & \qquad\qquad - h^\gamma\det B_{2\times 2}
+ 2h^{(\alpha+\gamma)/2} \big\langle B_{2\times 2}: \mbox{cof}\,\nabla' (S_{31}, S_{32})\big\rangle, \\
& R_{12, 13}\simeq - h^{\alpha/2}\partial_1
\mbox{curl}(S_{31}, S_{32}) + h^{\gamma/2} \big(\mbox{curl}\,B_{2\times 2} \big)_1,\\
& R_{12, 23}\simeq - h^{\alpha/2}\partial_2
\mbox{curl}(S_{31}, S_{32}) + h^{\gamma/2} \big(\mbox{curl}\,B_{2\times 2} \big)_2,
\\
& \qquad R_{13, 13}\simeq - h^{\alpha/2} \partial^2_{11} S_{33} +
2h^{\gamma/2} \partial_1 B_{13},\\
& \qquad R_{13, 23}\simeq - h^{\alpha/2}\partial^2_{12} S_{33} +
h^{\gamma/2}\big(\partial_1 B_{23}+\partial_2 B_{13}\big),\\
& \qquad R_{23, 23}\simeq - h^{\alpha/2}\partial^2_{22} S_{33} + 2h^{\gamma/2}\partial_2 B_{23}.\\
\end{split}
\end{equation}
Again, the optimality conditions and the energy scalings in all cases
of Theorem \ref{th_main150}, can be read from the first three
curvatures above.   When $\alpha=2$, $\gamma>2$ we have
$R_{12,12}\simeq h^2\big(-3 \det\nabla'(S_{31}, S_{32})
+ \big\langle\nabla\,\mbox{curl}(S_{31}, S_{32}), (S_{31}, S_{32})^\perp\big\rangle\big)$,
$(R_{12, 13}, R_{12, 23}) \simeq -h\nabla'\mbox{curl}(S_{31}, S_{32})$
and the fact that at least one of  these expressions is nonzero is
equivalent to the condition displayed in case (i). When
$\alpha=\gamma=2$, then get the full expressions:  
$R_{12,12}\simeq h^2\big(-3 \det\nabla'(S_{31}, S_{32})
+ \big\langle\nabla\,\mbox{curl}(S_{31}, S_{32}), (S_{31},
S_{32})^\perp\big\rangle\big) - \det B_{2\times 2}
+ 2 \big\langle B_{2\times 2}: \mbox{cof}\,\nabla' (S_{31}, S_{32})\big\rangle\big)$
and $(R_{12, 13}, R_{12, 23}) \simeq h\big(\mbox{curl}\, B_{2\times 2}
-\nabla'\mbox{curl}(S_{31}, S_{32})\big)$, which are again consistent with (i). In both cases $\inf I^h$ is of
order equal to the square of the stretching order, i.e. $h^4$.
For $\gamma=2$, $\alpha>2$ we get: $R_{12,12}\simeq -h^2\det B_{2\times 2}$
and $(R_{12, 13}, R_{12, 23}) \simeq h\,\mbox{curl}\, B_{2\times 2}$,
in agreement with case (ii). Finally, when $\alpha, \gamma>2$ then
$(R_{12, 13}, R_{12, 23})$ is of the order $h^{(\alpha\wedge\gamma)/2}$, with the corresponding coefficients 
indicated in (iii). The compatible stretching order in that case is
$h^{1+(\alpha\wedge\gamma)/2}$ which is strictly larger than the order
$h^{\alpha\wedge\gamma\wedge (\alpha+\gamma)/2}$ of $R_{12,12}$, as
computed in (\ref{curv_2}). Thus, stretching-related curvature does
not contribute towards the energy, and the exponent of the order of
$\inf I^h$ equals twice the bending order plus $1$, which yields $h^{2+\alpha\wedge\gamma}$.

We conjecture that the next viable energy scalings after those
displayed in Theorem \ref{th_main150} are guided by the lowest order
terms in $R_{13, 13}, R_{13, 23}, R_{23, 23}$ given in (\ref{curv_2}).


\subsection{An outline} We prove the convergence properties
(\ref{pomoc}) and the lower bounds of both Theorems \ref{th_main1} and
\ref{th_main15} in section \ref{sec_lb}. The proofs are written in a
unified manner and are only specified to the two assumed prestrain scalings
in Corollaries \ref{cor_com3} and \ref{cor_S0}. Section \ref{sec_recov}
gathers all constructions of recovery sequence in this papers: in sections \ref{standard}
and \ref{addi} we provide the upper bounds portion of the
$\Gamma$-limit statements in Theorems \ref{th_main1} and
\ref{th_main15}, while section \ref{energy_rough} is devoted to showing
Theorem \ref{th_general_scaling}. Corollaries \ref{coro1} and
\ref{coro2} rely on two decomposition results for symmetric matrix
fields in section \ref{sec_Helm}, where we also prove Proposition
\ref{infbd}. The optimality conditions listed in Theorems
\ref{th_main10} and \ref{th_main150} follow from the same
decomposition results, while the identification in terms of Riemann
curvatures in (\ref{curv_1}) and (\ref{curv_2}) is given in section \ref{sec_curv}.

\subsection*{Acknowledgements} This research begun in the framework of
Women in Mathematics of Materials (WIMM) Workshop. The authors are
grateful to the Association for Women in Mathematics and the Michigan
Center for Applied and Interdisciplinary Mathematics for their support
and hospitality. Silvia Jimenez-Bolanos was supported by the Picker
Fellowship from the Colgate University Research Council. 
Marta Lewicka was supported by NSF grants DMS-1613153
and DMS-2006439.

\section{Compactness and lower bounds}\label{sec_lb}

In this section, we investigate the asymptotic properties of minimizing
sequences to $I^h$ in the scaling sub-regimes of: $\lim_{h\to 0}h^{-2}\inf I^h=0$. First,
it  follows from \cite[Theorem 1.6]{lemapa1} that:
\begin{lemma}\label{lem1}
Let $\delta>0$ and assume that for a sequence $\{u^h\in
W^{1,2}(\Omega^h,\R^3)\}_{h\to 0}$ there holds: $I^h(u^h)\leq Ch^{2+\delta}$. Then there
exists a sequence $\{R^h\in W^{1,2}(\omega, SO(3))\}_{h\to 0}$ such that:
\begin{equation}\label{app}
\frac{1}{h}\int_{\Omega^h}|\nabla u^h - R^hA^h|^2\dx\leq C h^{2+ \delta\wedge \alpha\wedge\gamma},
\qquad  \int_{\omega} |\nabla R^h|^2\dxp\leq C h^{ \delta\wedge \alpha\wedge\gamma}.
\end{equation}
\end{lemma}
Second, in the context of Lemma \ref{lem1}, the scaled integrands in
${h^{-(2+\delta)}}I^h(u^h)$, equal in virtue of frame invariance and by Taylor expanding $W$:
\begin{equation}\label{expa}
\frac{1}{h^{2+\delta}} W\big((R^h)^T\nabla u^h (A^h)^{-1}\big) =
\frac{1}{2}D^2W(Id_3) (G^h)^{\otimes 2} + o(|G^h|^2) \quad \mbox{ as }\; h\to 0,
\end{equation}
where we define the sequence of strains $\{G^h\in L^2(\Omega^1,\mathbb{R}^{3\times 3})\}_{h\to 0}$ in:
$$G^h(x', x_3) = \frac{1}{h^{1+\delta/2}}\Big( R^h(x')^T\nabla u^h(x',
hx_3) A^h(x', hx_3)^{-1}-Id_3\Big)\quad\mbox{ for all } x'\in \omega,
\; x_3\in \big(-\frac{1}{2}, \frac{1}{2}\big).$$
In view of the expansion (\ref{expa}), it is natural to expect that
the limit of $h^{-(2+\delta)} I^h(u^h)$ quantifies the limit of
$\{G^h\}_{h\to 0}$. By (\ref{app}), a sufficient condition to get
a subsequential convergence in:
\begin{equation}\label{due}
G^h\rightharpoonup G\qquad \mbox{ weakly in}\; L^2(\Omega^1,\mathbb{R}^{3\times
  3}), \mbox{ as }\; h\to 0
\end{equation}
is thus: $1+(\delta\wedge\alpha\wedge\gamma)/2\geq 1+\delta/2$, or equivalently:
$$\delta \leq \alpha\wedge\gamma.$$

We have the following compactness result:
\begin{lemma}\label{lem_com}
Let $\delta\in (0, \alpha\wedge\gamma]$ and assume that for a sequence $\{u^h\in
W^{1,2}(\Omega^h,\R^3)\}_{h\to 0}$ there holds: $I^h(u^h)\leq
Ch^{2+\delta}$. Then there exist sequences $\{\bar R^h\in
SO(3)\}_{h\to 0}$, $\{c^h\in \R^3\}_{h\to 0}$ such that there holds for the rescaled deformations
$y^h(x', x_3) = (\bar R^h)^Tu^h(x', hx_3)-c^h\in W^{1,2}(\Omega^1, \R^3)$:
\begin{enumerate}[leftmargin=7mm]
\item[(i)] $y^h\to x'$ strongly in $W^{1,2}(\Omega^1, \R^3)$, as $h\to 0$.\vspace{1mm}
\item[(ii)] Define: $V^h(x')=h^{-\delta/2} \fint_{-1/2}^{1/2}y^h(x',
  x_3)-x'\;\mathrm{d}x_3$. Then, up to a subsequence: $V^h\to V$ as
  $h\to 0$ strongly in $W^{1,2}(\omega, \R^3)$. The limit $V\in
 W^{2,2}(\omega, \R^3)$ and it satisfies:
$$\mathrm{sym}(\nabla V)_{2\times 2}= \left\{\begin{array}{ll} 0 &\mbox{for }\;
    \delta<\alpha \\ S_{2\times 2} &\mbox{for }\; \delta=\alpha. \end{array}\right.$$
In particular, for $\delta<\alpha$ or when $S_{2\times 2}=0$, we get that
$V=(0,0,v)$ for an out-of-plane displacement $v\in W^{2,2}(\omega,\R)$. 
The case $\delta=\alpha$ is viable only when
$\mathrm{curl}^T\mathrm{curl}\,S_{2\times 2} \equiv 0$ in $\omega$.
\item[(iii)] The in-plane part of the limiting strain $G$ in (\ref{due}) satisfies:
\begin{equation*}
\begin{split}
\big(\partial_3 G\big)_{2\times 2} = -\nabla^2V_3
& + \left\{\begin{array}{ll} 0 &\mbox{for }\; \delta<\alpha\\ 2 \big(\nabla (S_{31}, S_{32})\big)_{2\times 2}
&\mbox{for }\; \delta=\alpha\end{array}\right. - 
\left\{\begin{array}{ll} 0 &\mbox{for }\; \delta<\gamma\\ B_{2\times 2}
&\mbox{for }\; \delta=\gamma.\end{array}\right.
\end{split}
\end{equation*}
Hence $G(x',0)_{2\times 2}$ is well defined, and it equals:
$\fint_{-1/2}^{1/2}G(x', x_3)_{2\times 2}\;\mathrm{d}x_3\in L^{2}(\omega,\R^{2\times 2}).$
\end{enumerate}
\end{lemma}
\begin{proof}
{\bf 1.} As in the proof of \cite[Theorem 1.2]{lemapa1}, we define:
$$\bar R^h = \tilde R^h\hat R^h \quad \mbox{where: } \tilde
R^h=\mathbb{P}_{SO(3)}\fint_\omega R^h\dxp, \quad \hat
R^h=\mathbb{P}_{SO(3)}\fint_{\Omega^h} (\tilde R^h)^T\nabla u^h\dx,$$
where the uniqueness of the above projections follows in virtue of the
bounds in (\ref{app}). Further:
\begin{equation}\label{uno}
\|(\bar R^h)^TR^h-Id_3\|_{W^{1,2}(\omega)}\leq C h^{\delta/2}.
\end{equation}
By choosing $c^h$ so that $\fint_{\Omega^1}y^h\dx=\fint_\omega
x'\dxp$, we obtain the following bounds that imply (i):
$$\fint_{\Omega^1}|(\nabla y^h-Id)_{3\times 2}|^2\dx \leq Ch^\delta, 
\qquad \fint_{\Omega^1}|\partial_3y^h|^2\dx \leq Ch^2.$$

\smallskip

{\bf 2.} Fix $s\ll 1$. From (\ref{uno}), (\ref{due}) we observe the subsequential convergence as $h\to 0$, in:
\begin{equation}\label{tre}
\begin{split}
\frac{1}{s}\big(G(x', &x_3+s) - G(x', x_3)\big)_{2\times
  2} \\ & \leftharpoonup \frac{1}{s}\Big((\bar R^h)^TR^h \Big(G^h(x', x_3+s) A^h(x',
hx_3+hs) - G^h(x',x_s)A^h(x', hx_3)\Big)\Big)_{2\times 2} \\ & =
\frac{1}{s} \frac{1}{h^{1+\delta/2}}\Big(\nabla y^h(x', x_3+s) -
\nabla y^h(x', x_3)\Big)_{2\times 2} - h^{\gamma/2 - \delta/2} B(x')_{2\times 2}.
\end{split}
\end{equation}
To identify the limit of the first term in the right hand side above,
we consider the asymptotics of:
$h^{-(1+\delta/2)}\frac{1}{s} \big(y^h(x', x_3+s) - y^h(x', x_3)\big)
= h^{-(1+\delta/2)}\fint_0^s\partial_3y^h(x', x_3+t)\;\mbox{d}t$. Namely, we write: 
\begin{equation*}
\begin{split}
&\frac{1}{h^{1+\delta/2}}\big(\partial_3y^h(x',x_3) - he_3\big) = \frac{1}{h^{\delta/2}}
\Big((\bar R^h)^T\partial_3u^h(x', hx_3) - e_3\pm A^h(x',
hx_3)e_3\Big) \\ & \qquad = \frac{1}{h^{\delta/2}} (\bar R^h)^T\Big(\nabla u^h(x', hx_3) - R^h A^h(x',
hx_3)\Big) e_3 + \frac{1}{h^{\delta/2}}\Big((\bar
R^h)^TR^h(x')-Id_3\Big)A^h(x', hx_3)e_3 \vspace{1mm} \\ & \qquad \qquad + h^{\alpha/2-\delta/2}S(x')
e_3 + h^{\gamma/2+1-\delta/2}x_3B(x')e_3.
\end{split}
\end{equation*}
The first term in the right hand side above is bounded by $Ch$ from (\ref{app}), so it
converges to $0$. For the second term, we first denote the following subsequential limit:
\begin{equation}\label{quattro}
\frac{1}{h^{\delta/2}}\Big((\bar R^h)^TR^h-Id_3\Big)\rightharpoonup P \quad \mbox{ weakly in}\;
W^{1,2}(\omega,\mathbb{R}^{3\times 3}), \mbox{ as }\; h\to 0,
\end{equation}
whose existence is implied by (\ref{uno}) and where the limiting field
satisfies: $P\in W^{1,2}(\omega, so(3))$.
Recalling (\ref{tre}) we hence obtain:
\begin{equation*}
\begin{split}
G(x', x_3)_{2\times 2} - G(x', 0)_{2\times 2} =  x_3 \big(\nabla (Pe_3)(x')\big)_{2\times 2}
& + x_3 \left\{\begin{array}{ll} 0 &\mbox{for }\; \delta<\alpha\\ \big(\nabla (Se_3)(x')\big)_{2\times 2}
&\mbox{for }\; \delta=\alpha\end{array}\right. \\ & - 
x_3 \left\{\begin{array}{ll} 0 &\mbox{for }\; \delta<\gamma\\ B(x')_{2\times 2}
&\mbox{for }\; \delta=\gamma.\end{array}\right.
\end{split}
\end{equation*}

\smallskip

{\bf 3.} We now identify the entries $P_{31}, P_{32}$ of the limiting skew field in (\ref{quattro}). We write:
\begin{equation*}
\begin{split}
\nabla V^h = & \; \frac{1}{h^{\delta/2}}(\bar R^h)^T\fint_{-h/2}^{h/2}\big(\nabla u^h - R^hA^h\big)_{3\times
  2}\;\mbox{d}x_3 + \frac{1}{h^{\delta/2}}\big((\bar
R^h)^TR^h - Id_3\big)_{3\times 2} \\ & + h^{\alpha/2-\delta/2}(\bar R^h)^T R^h S_{3\times 2}, 
\end{split}
\end{equation*}
so by (\ref{quattro}) we get, up to a subsequence:
\begin{equation*}
\nabla V^h\to P_{3\times 2} + \left\{\begin{array}{ll} 0 &\mbox{for
    }\; \delta<\alpha\\ S_{3\times 2}  &\mbox{for }\; \delta=\alpha\end{array}\right.  
 \quad \mbox{ strongly in}\; L^{2}(\omega,\mathbb{R}^{3\times 2}), \mbox{ as }\; h\to 0.
\end{equation*}
Since $\fint_\omega V^h\dxp =0$, this proves the convergence statement
in (ii), with $\nabla V$ given by the right hand side
above. Equivalently, we record the useful formula:
\begin{equation}\label{cinque}
P_{3\times 2} = \nabla V -  \left\{\begin{array}{ll} 0 &\mbox{for
    }\; \delta<\alpha\\ S_{3\times 2}  &\mbox{for }\; \delta=\alpha.\end{array}\right.  
\end{equation}

Note now that the definitions of $\tilde R^h$ and $\hat R^h$ imply
that: $\mbox{skew} \fint_\omega(\nabla V)_{2\times 2} \dxp= 0$, because:
\begin{equation}\label{cinque5}
\fint_\omega (\nabla V^h)_{2\times 2}\dxp
=\frac{1}{h^{\delta/2}}\Big((\hat R^h)^T\fint_{\Omega^h}(\tilde
R^h)^T\nabla u^h-\hat R^h\dx\Big)_{2\times 2}\in \R^{2\times 2}_{sym}.
\end{equation}
There also holds: $\fint_\omega V\dxp=0$.
Consequently, in case $\delta<\alpha$ or $S_{2\times 2}=0$ it follows that the tangential
component of $V$ must have gradient $0$ and mean $0$, and hence be equal to $0$. 
This completes the proof of (ii), while from (\ref{cinque}) we get:
\begin{equation}\label{sei}
\big((Pe_3)_1, (Pe_3)_2\big) = -(P_{31}, P_{32}) =
-(\partial_1V_3,\partial_2V_3) + \left\{\begin{array}{ll} 0 &\mbox{for
    }\; \delta<\alpha\\ (S_{31}, S_{32})  &\mbox{for }\; \delta=\alpha.\end{array}\right.  
\end{equation}
The formula in (iii) is now a consequence of step 2.
\end{proof}

\begin{corollary}\label{low_bd}
In the context of Lemma \ref{lem_com}, there holds:
\begin{equation*}
\liminf_{h\to 0}\frac{1}{h^{2+\delta}} I^h(u^h) \geq
\frac{1}{2}\int_\omega\mathcal{Q}_2\big(\mathrm{sym}\,G(x',0)_{2\times 2}\big)\dxp +
\frac{1}{24}\int_{\omega}\mathcal{Q}_2 \big(\mathrm{sym}\,\partial_3 G_{2\times 2} \big)\dxp,
\end{equation*}
where the quadratic form $\mathcal{Q}_2$ is defined in:
\begin{equation*}
\mathcal{Q}_{2}(F_{2\times 2})=\min  \left\{ \mathcal{Q}_3(\tilde{F}); ~~
\tilde{F}\in\mathbb{R}^{3\times 3}\;\; \mbox{with} \;\;\tilde{F}_{2\times 2}=F_{2\times 2} \right\}, 
\quad \mathcal{Q}_{3}(F)=D^2W({Id}_3)(F,F).
\end{equation*}
\end{corollary}
\begin{proof}
Recalling (\ref{expa}) and the fact that $\{G^h\}_{h\to
  0}$ is bounded in $L^2(\Omega^1,\R^3)$, we write:
$$\frac{1}{h^{2+\delta}} I^h(u^h)\geq \int_{\{h^{\delta/2}|G^h|\leq
  1\}}\frac{1}{2}\mathcal{Q}_3(G^h) + o(1)|G^h|^2\dx =
\frac{1}{2}\int_{\Omega^1}\mathcal{Q}_3\big(\mathds{1}_{\{h^{\delta/2}|G^h|\leq
1\}}|G^h|\big)\dx+o(1),$$
as $h\to 0$. Further, since the argument in $\mathcal{Q}_3$ in the
right hand side above converges, weakly in $L^2(\Omega^1, \R^{3\times
  3})$, to $G$ by (\ref{due}), the weak lowersemicontinuity argument implies that:
$$\liminf_{h\to 0}\frac{1}{h^{2+\delta}} I^h(u^h)\geq \frac{1}{2}\int_{\Omega^1}\mathcal{Q}_3(G)\dx \geq 
\frac{1}{2}\int_{\Omega^1}\mathcal{Q}_3(\mathrm{sym}\,G)\dx \geq 
\frac{1}{2}\int_{\Omega^1}\mathcal{Q}_2(\mathrm{sym}\,G_{2\times
  2})\dx. $$
Here, we also used that
$\mathcal{Q}_3(G)=\mathcal{Q}_3(\mathrm{sym}\,G)$ and the definition
of $\mathcal{Q}_2$. By Lemma \ref{lem_com} (iii) we get: $G(x', x_3)_{2\times 2} =
G(x', 0)_{2\times 2} + x_3\partial_3 G(x', \cdot)_{2\times 2}$, which
concludes the proof. 
\end{proof}

\begin{lemma}\label{lem_com2}
In the context of Lemma \ref{lem_com}, we have the following
subsequential convergence,  weakly in $L^2(\omega, \R^{2\times 2})$ as $h\to 0$:
\begin{equation}\label{sette}
\begin{split}
\mathrm{sym}\,G(\cdot, 0)_{2\times 2}\leftharpoonup & \; \frac{1}{h}\mathrm{sym} \big(\nabla V^h\big)_{2\times 2} -
\frac{1}{h^{1+\delta/2}}\,\mathrm{sym}\big((\bar R^h)^TR^h-Id_3\big)_{2\times 2} 
\\ & - h^{\alpha/2-1-\delta/2}\, \mathrm{sym}\big((\bar R^h)^TR^h S\big)_{2\times 2}.
\end{split}
\end{equation}
\end{lemma}
\begin{proof}
We write:
\begin{equation*}
\begin{split}
\fint_{-1/2}^{1/2}G^h_{2\times 2}\;\mathrm{d}x_3 = \; & \Big(\big((\bar
R^h)^TR^h - Id_3\big)^T (\bar R^h)^T\cdot \frac{1}{h^{1+\delta/2}}\fint_{-h/2}^{h/2}(\nabla u^h -
R^hA^h)\;\mbox{d}x_3\cdot (A^h)^{-1}\Big)_{2\times 2} \\ & + 
\Big((\bar R^h)^T\cdot \frac{1}{h^{1+\delta/2}}\fint_{-h/2}^{h/2}(\nabla u^h -
R^hA^h)\;\mbox{d}x_3\Big)_{2\times 2}  \\ & + \Big((\bar
R^h)^T \cdot \frac{1}{h^{1+\delta/2}}\fint_{-h/2}^{h/2}(\nabla u^h -
R^hA^h)\;\mbox{d}x_3\cdot \big((A^h)^{-1}-Id_3\big)\Big)_{2\times 2}.
\end{split}
\end{equation*}
The first and the third terms in the right hand side above converge to $0$, in
$L^2(\omega, \R^{2\times 2})$ as $h\to 0$, because of (\ref{uno}) and (\ref{app}).
We rewrite the second term as:
$$\frac{1}{h} \big(\nabla V^h\big)_{2\times 2} -
\frac{1}{h^{1+\delta/2}}\big((\bar R^h)^TR^h-Id_3\big)_{2\times 2} 
- h^{\alpha/2-1-\delta/2}\big((\bar R^h)^TR^h S\big)_{2\times 2}.$$
This implies the result, because of the subsequential convergence
weakly in $L^2(\omega, \R^{2\times 2})$, of
$\fint_{-1/2}^{1/2}G^h(\cdot, x_3)_{2\times 2}\;\mathrm{d}x_3 $ to
$\fint_{-1/2}^{1/2}G(\cdot, x_3)_{2\times 2}\;\mathrm{d}x_3 = G(\cdot, 0)_{2\times 2}$.
\end{proof}

\begin{corollary}\label{cor_com3}
In the context of Lemma \ref{lem_com}, assume further that $\delta\geq
2$ and $\alpha\geq 2+\delta$. Then, $h^{-1}(V^h_1, V^h_2)$ converges
to some in-plane displacement $w$, up to a subsequence weakly in $W^{1,2}(\omega,
\R^{2})$ as $h\to 0$. Moreover, there holds: 
\begin{equation*}
\mathrm{sym}\,G(\cdot, 0)_{2\times 2} = \; \mathrm{sym} \nabla w  + \left\{\begin{array}{ll} 0 &\mbox{for
    }\; \delta>2\\ \frac{1}{2} (\nabla v)^{\otimes 2} &\mbox{for }\; \delta=2\end{array}\right. -
\left\{\begin{array}{ll} 0 &\mbox{for }\; \alpha>2+\delta \\ S_{2\times 2}
&\mbox{for }\; \alpha = 2+\delta.\end{array}\right.
\end{equation*}
Consequently, $\liminf_{h\to 0} h^{-(2+\delta)}I^h(u^h)$ is bounded from below by:
\begin{equation}\label{lowbdbd}
\begin{split}
& I(v,w) =  \frac{1}{24}\int_\omega \mathcal{Q}_2\Big(\nabla^2 v +
\left\{\begin{array}{ll} 0 & \mbox{for } \; \delta<\gamma\\ B_{2\times
      2} &\mbox{for } \; \delta=\gamma\end{array}\right. \Big) \dxp \\
& \quad +\frac{1}{2} \int_\omega \mathcal{Q}_2\Big(\mathrm{sym}\,\nabla w + 
\left\{\begin{array}{ll} 0 & \mbox{for } \; \delta>2\\ \frac{1}{2} (\nabla
    v)^{\otimes 2} &\mbox{for } \; \delta=2\end{array}\right. 
- \left\{\begin{array}{ll} 0 & \mbox{for } \; \alpha>2+\delta \\ S_{2\times
      2} &\mbox{for } \; \alpha=2+\delta \end{array}\right. \Big)\dxp,
\end{split}
\end{equation}
which coincides with the  functionals given in Theorem \ref{th_main1} in each of the indicated three cases.
\end{corollary}
\begin{proof}
The second term in the right hand side of (\ref{sette}) can be
rewritten with the help of:
$$  -\mathrm{sym}\big((\bar R^h)^TR^h-Id_3\big)_{2\times 2}
 = \frac{1}{2} \mathrm{sym} \Big(\big((\bar R^h)^TR^h-Id_3\big)^T \big((\bar
R^h)^TR^h-Id_3\big)\Big)_{2\times 2}$$
By (\ref{quattro}) and since $\delta\geq 1+\delta/2$, we conclude
the following subsequential convergence, strongly in $L^{2}(\omega, \R^{2\times 2})$ as $h\to 0$:
$$  -\frac{1}{h^{1+\delta/2}}\,\mathrm{sym}\big((\bar R^h)^TR^h-Id_3\big)_{2\times 2} 
\rightharpoonup \left\{\begin{array}{ll} 0 &\mbox{for
    }\; \delta>2\\ \frac{1}{2} (P^TP)_{2\times 2} 
&\mbox{for }\; \delta=2.\end{array}\right. $$
When $\alpha\geq 2+\delta$, then the third term in (\ref{sette})
converges, by (\ref{uno}), to $0$ for $\alpha>2+\delta$ and to
$-S_{2\times 2}$ for $\alpha=2+\delta$. In this case, we also
have $V=(0,0,v)$ and (\ref{cinque}) yields that $(P^TP)_{2\times
  2}=(\nabla v)^{\otimes 2}$. This proves the formula for $G(\cdot,
0)_{2\times 2}$, since the resulting weak subsequential convergence of the first term
in the right hand side of (\ref{sette}):
$\frac{1}{h}\mathrm{sym} \big(\nabla V^h\big)_{2\times 2}$, is
equivalent with the weak convergence of  $h^{-1}(V^h_1, V^h_2)$ to some
limiting $w\in W^{1,2}(\omega, \R^2)$.

Finally, the lower bound on $\liminf_{h\to 0} h^{-(2+\delta)}I^h$ is derived from Corollary \ref{low_bd},
upon recalling Lemma \ref{lem_com} (iii). The optimal values of the exponent
$\delta$ given in Theorem \ref{th_main1} in function of $\alpha\geq 4$
and $\gamma\geq 2$, follow by a direct inspection.
\end{proof}

\begin{corollary}\label{cor_S0}
In the context of Lemma \ref{lem_com}, assume further that $\delta\geq
2$ and $S_{2\times 2}\equiv 0$ in $\omega$. Then,  $h^{-1}(V^h_1, V^h_2)$ converges
to some in-plane displacement $w$, up to a subsequence weakly in $W^{1,2}(\omega,
\R^{2})$ as $h\to 0$. Moreover, there holds: 
\begin{equation}\label{sette5}
\mathrm{sym}\,G(\cdot, 0)_{2\times 2} = \mathrm{sym} \nabla w  + \left\{\begin{array}{ll} 0 &\mbox{for
    }\; \delta>2\\ \frac{1}{2} (\nabla v)^{\otimes 2}
&\mbox{for }\; \delta=2<\alpha \vspace{1mm} \\  \frac{1}{2} (\nabla
v)^{\otimes 2} - \frac{1}{2}(S_{31}, S_{32})^{\otimes 2} &\mbox{for }\; \delta=\alpha = 2.\end{array}\right.
\end{equation}
Consequently, $\liminf_{h\to 0} h^{-(2+\delta)}I^h(u^h)$ is bounded from below by:
\begin{equation}\label{lowbdbd5}
\begin{split}
& I(v,w) =  \frac{1}{24}\int_\omega \mathcal{Q}_2\Big(\nabla^2 v - 
\left\{\begin{array}{ll} 0 &\mbox{for }\; \delta<\alpha\\ 2 \,\mathrm{sym}\big(\nabla (S_{31}, S_{32})\big)_{2\times 2}
&\mbox{for }\; \delta=\alpha\end{array}\right. +
\left\{\begin{array}{ll} 0 & \mbox{for } \; \delta<\gamma\\ B_{2\times
      2} &\mbox{for } \; \delta=\gamma\end{array}\right. \Big) \dxp \vspace{1mm} \\
& \quad +\frac{1}{2} \int_\omega \mathcal{Q}_2\Big(\mathrm{sym}\,\nabla w + 
\left\{\begin{array}{ll} 0 & \mbox{for } \; \delta>2\\ \frac{1}{2} (\nabla
    v)^{\otimes 2} &\mbox{for } \; \delta=2\end{array}\right. 
- \left\{\begin{array}{ll} 0 & \mbox{for } \; \alpha>2\\ 
\frac{1}{2}(S_{31}, S_{32})^{\otimes 2} &\mbox{for } \; \alpha=2\end{array}\right. \Big)\dxp,
\end{split}
\end{equation}
which coincides with the  functionals given in Theorem \ref{th_main15} in each of the indicated cases.
\end{corollary}
\begin{proof}
Convergence of the second term in the right hand side of (\ref{sette})
follows as in the proof of Corollary \ref{cor_com3}.
To show (\ref{sette5}), observe that the third term in (\ref{sette}) can be now written as:
\begin{equation}\label{sette7}
- h^{\alpha/2-1}\, \mathrm{sym}\Big(\frac{(\bar R^h)^TR^h -Id_3}{h^{\delta/2}}\cdot S\Big)_{2\times 2},
\end{equation}
and we may identify its limit when $\alpha\geq 2$, in virtue of (\ref{quattro}).
Namely, when $\alpha>2$ this limit is $0$ and we recover the first two
cases of (\ref{sette5}). When $\alpha=2$ then automatically $\delta=2$
as well, and (\ref{sette7}) converges subsequentially, weakly in
$L^2(\omega, \R^{2\times 2})$ as $h\to 0$ to:
$$-\mathrm{sym}\big(P S\big)_{2\times 2} = \mathrm{sym}\Big( \big(\nabla v - (S_{31},
S_{32})\big) \otimes (S_{31}, S_{32})\Big) = \mathrm{sym}\big(\nabla v \otimes (S_{31},
S_{32})\big) - (S_{31}, S_{32})^{\otimes 2}$$
where we used (\ref{cinque}). At the same time: $\frac{1}{2}(P^TP)_{2\times
  2} = \frac{1}{2}\big(\nabla v - (S_{31}, S_{32})\big)^{\otimes 2}$,
which concludes the proof of the last case in (\ref{sette5}).

The lower bound in (\ref{lowbdbd5}) follows from Corollary \ref{low_bd},
Lemma \ref{lem_com} (iii) and (\ref{sette5}). The optimal values of 
$\delta$ in function of $\alpha, \gamma\geq 2$, follow by a direct inspection.
\end{proof}

\section{Recovery sequences}\label{sec_recov}

In this section, we construct sequences of deformations $\{u^h\}_{h\to 0}$ 
with the desired asymptotics of the energy $I^h(u^h)$. First, in
subsection \ref{energy_rough} we prove the upper bound stated in Theorem \ref{th_general_scaling}.
Second, we complete the $\Gamma$-convergence results given in
Theorems \ref{th_main1} and \ref{th_main15}, by constructing recovery sequences for the
general form of the limiting functional $I$. In subsection \ref{standard}
we treat the case (\ref{lowbdbd}), and in subsection \ref{addi} the case (\ref{lowbdbd5}).

\smallskip

\subsection{Case $\mathbf{0<\alpha<4}$ and $\mathbf{\gamma>0}$}\label{energy_rough}
Recall that by \cite[Theorem 1.1]{CS}, for every $a\in (0,1/5)$ there exists $v\in\mathcal{C}^{1,a}(\bar\omega,\R)$
and $w\in \mathcal{C}^{1,a}(\bar\omega,\R^2)$ satisfying:
\begin{equation}\label{exactMA}
\frac{1}{2}(\nabla v)^{\otimes 2} + \mathrm{sym}\,\nabla w = S_{2\times 2}.
\end{equation}
We now regularize $v,w$ to
$v_\epsilon\in\mathcal{C}^\infty(\bar\omega,\R)$, $w_\epsilon\in \mathcal{C}^\infty(\bar\omega,\R^2)$, 
by means of the family of standard convolution kernels $\{\phi_\epsilon(x) =
\epsilon^{-2}\phi(x/\epsilon)\}_{\epsilon\to 0}$ where $\epsilon$ is a power of $h$ to be chosen later:
$$ v_\epsilon= v*\phi_\epsilon,\quad w_\epsilon=w*\phi_\epsilon \quad \mbox{
and } \quad \epsilon=h^t.$$
We will utilize the following bound, resulting from the commutator estimate \cite[Lemma 1]{CDS}:
\begin{equation}\label{m1}
\begin{split}
\big\| \frac{1}{2}(\nabla v_\epsilon&)^{\otimes 2} +
\mathrm{sym}\,\nabla w_\epsilon - S_{2\times
  2}\big\|_{\mathcal{C}^0(\omega)} \\ & \leq \big\| \frac{1}{2}(\nabla v_\epsilon)^{\otimes 2} +
\mathrm{sym}\,\nabla w_\epsilon - S_{2\times
  2}*\phi_\epsilon\big\|_{\mathcal{C}^0(\omega)} + \| S_{2\times 2}*\phi_\epsilon- S_{2\times 2}\|_{\mathcal{C}^0(\omega)} 
\\ & = \frac{1}{2}\big\| (\nabla v_\epsilon)^{\otimes 2} - (\nabla
v)^{\otimes 2}*\phi_\epsilon\big\|_{\mathcal{C}^0(\omega)} + \|
S_{2\times 2}*\phi_\epsilon- S_{2\times 2}\|_{\mathcal{C}^0(\omega)}
\\ & \leq C\epsilon^{2a} + C\epsilon^2 \leq C\epsilon^{2a},
\end{split}
\end{equation}
where the $C\epsilon^2$ bound follows by Taylor expanding $S_{2\times
  2}$ up to second order terms. Further, by \cite[(4.5)]{LP}, we get the uniform bounds:
\begin{equation}\label{m2}
\|\nabla v_\epsilon\|_{\mathcal{C}^0(\omega)} +
\|\nabla w_\epsilon\|_{\mathcal{C}^0(\omega)} \leq C,\qquad 
\|\nabla^2v_\epsilon\|_{\mathcal{C}^0(\omega)} +
\|\nabla^2w_\epsilon\|_{\mathcal{C}^0(\omega)} \leq C\epsilon^{a-1}.
\end{equation}

\medskip

\textbullet $~~$ Let $\delta=\alpha/2$. We denote $s=(S_{31}, S_{32})^T$ and define
$u^h\in \mathcal{C}^\infty(\bar \Omega^h,\R^3)$ by: 
\begin{equation*}\label{recseq0}
u^h = id_3 + h^{\delta/2}\left[\begin{array}{c} 0\\
    v_\epsilon\end{array}\right] + h^{\delta}\left[\begin{array}{c} w_\epsilon\\ 0\end{array}\right] 
+x_3\Big(h^{\delta/2} \left[\begin{array}{c} -\nabla v_\epsilon\\ 0\end{array}\right]
+ h^{\delta} \left[\begin{array}{c} 2 s\\ S_{33}-\frac{1}{2}|\nabla v_\epsilon|^2\end{array}\right]
+h^{3\delta/2}\vec b\Big),
\end{equation*}
with the higher order smooth correction vector field:
$$\vec b(x') = \Big(-S_{33}\nabla v_\epsilon+\frac{1}{2}|\nabla
v_\epsilon|^2\nabla v_\epsilon + (\nabla w_\epsilon)^T\nabla
v_\epsilon, 2\langle \nabla v_\epsilon, s\rangle\Big)^T.$$
It follows that:
\begin{equation*}
\begin{split}
\nabla u^h(x', hx_3) = & \;Id_3 + h^{\delta/2} \left[\begin{array}{cc} 
   0& -\nabla v_\epsilon\\ \nabla v_\epsilon &0 \end{array}\right] +
h^\delta \left[\begin{array}{cc} \nabla w_\epsilon & 2s \\ 0 &
    S_{33}-\frac{1}{2}|\nabla v_\epsilon|^2\end{array}\right] +
h^{3\delta/2} \left[\begin{array}{ccc}   0& 0& \vec b \end{array}\right]\\ &
+ h^{1+\delta/2}x_3\left[\begin{array}{cc} -\nabla^2 v_\epsilon &0\\ 0&0\end{array}\right]
+\mathcal{O}\big(h^{1+\delta}+h^{1+3\delta/2}(\|\nabla^2v_\epsilon\|+\|\nabla^2w_\epsilon\|)\big).
\end{split}
\end{equation*}
We now write $(A^h(x', hx_3))^{-1}=Id_3 - h^\delta S -
h^{1+\gamma/2}x_3B + \mathcal{O}(h^{2\delta}+h^{2+\gamma})$ and proceed with:
\begin{equation*}
\begin{split}
\nabla& u^h(A^h)^{-1}(x', hx_3) = Id_3 + P^h +
h^\delta \left[\begin{array}{cc} \nabla w_\epsilon -S_{2\times 2} & 0  \\ 0 &
-\frac{1}{2}|\nabla v_\epsilon|^2\end{array}\right] \\ & +
h^{3\delta/2} \left[\begin{array}{cc} \nabla v_\epsilon \otimes s &
    \frac{1}{2}|\nabla v_\epsilon|^2\nabla v_\epsilon + (\nabla
    w_\epsilon)^T\nabla v_\epsilon\\ -S_{2\times 2}\nabla v_\epsilon &
\langle\nabla v_\epsilon, s\rangle\end{array}\right] 
+ x_3\Big( h^{1+\delta/2}\left[\begin{array}{cc} -\nabla^2 v_\epsilon &0\\ 0&0\end{array}\right]
- h^{1+\gamma/2}B\Big) \\ & 
+\mathcal{O}\big( h^{1+\delta}+h^{2\delta}+h^{1+\delta/2+\gamma/2}+ h^{2+\gamma}\big) 
+\mathcal{O}\big(h^{1+3\delta/2} + h^{2+\delta/2+\gamma/2} \big) (\|\nabla^2v_\epsilon\|+\|\nabla^2w_\epsilon\|).
\end{split}
\end{equation*}
Above, we used the following  skew-symmetric matrix field:
$$P^h = \left[\begin{array}{cc} 0 & p^h \\ -p^h & 0\end{array}\right],
\qquad p^h= - h^{\delta/2} \nabla v_\epsilon+h^\delta s. $$

\medskip

\textbullet $~~$ Consider the rotation fields $q^h\in\mathcal{C}^\infty(\bar \omega, SO(3))$:
$$q^h = \exp(-P^h) = Id_3 - P^h +\frac{1}{2}(P^h)^2-\frac{1}{6}(P^h)^3
+ \mathcal{O}(h^{2 \delta}). $$
Then we get:
\begin{equation*}
\begin{split}
q^h \nabla u^h(A^h)^{-1}&(x', hx_3) = Id_3 +
h^\delta \left[\begin{array}{cc} \frac{1}{2}(\nabla
    v_\epsilon)^{\otimes 2}+\nabla w_\epsilon -S_{2\times 2} & 0 \\ 0 &
0 \end{array}\right] \\ & + h^{3\delta/2} \left[\begin{array}{cc}
\mbox{skew}(\nabla v_\epsilon \otimes s) &  (\nabla
w_\epsilon)^T\nabla v_\epsilon\\ -  (\nabla w_\epsilon)^T\nabla
v_\epsilon & 0 \end{array}\right] +\frac{1}{3}(P^h)^{3} \\ & 
+ x_3\Big( h^{1+\delta/2}\left[\begin{array}{cc} -\nabla^2 v_\epsilon &0\\ 0&0\end{array}\right]
- h^{1+\gamma/2}B\Big) \\ & 
+\mathcal{O}\big( h^{1+\delta}+h^{2\delta}+h^{1+\delta/2+\gamma/2}+ h^{2+\gamma}\big) 
+\mathcal{O}\big(h^{1+\delta} (\|\nabla^2v_\epsilon\|+\|\nabla^2w_\epsilon\|)\big).
\end{split}
\end{equation*}
Finally, we apply a rotation field $\bar q^h = \exp(-\bar P^h) = Id_3 - \bar P^h + \mathcal{O}(h^{2
  \delta}) \in\mathcal{C}^\infty(\bar \omega, SO(3)),$ with:
$$\bar P^h = \left[\begin{array}{cc}
    \mbox{skew}\big(h^\delta\nabla w_\epsilon  + h^{3\delta/2}\nabla v_\epsilon \otimes s\big)
&  h^{3\delta/2} (\nabla w_\epsilon)^T\nabla v_\epsilon \\ - h^{3\delta/2} (\nabla w_\epsilon)^T\nabla v_\epsilon &
0 \end{array}\right] + \frac{1}{3}(P^h)^3,$$
to get:
\begin{equation*}
\begin{split}
\bar q^h q^h \nabla u^h(A^h)^{-1}&(x', hx_3) = Id_3 +
h^\delta \left[\begin{array}{cc} \frac{1}{2}(\nabla
 v_\epsilon)^{\otimes 2}+\mbox{sym}\nabla w_\epsilon -S_{2\times 2}  & 0 \\ 0 & 0 \end{array}\right] 
\\ & + x_3\Big( h^{1+\delta/2}\left[\begin{array}{cc} -\nabla^2 v_\epsilon &0\\ 0&0\end{array}\right]
- h^{1+\gamma/2}B\Big) \\ & 
+\mathcal{O}\big( h^{1+\delta}+h^{2\delta}+h^{1+\delta/2+\gamma/2}+ h^{2+\gamma}\big) 
+\mathcal{O}\big(h^{1+\delta} (\|\nabla^2v_\epsilon\|+\|\nabla^2w_\epsilon\|)\big).
\end{split}
\end{equation*}

\medskip

\textbullet $~~$ In conclusion, we obtain the following energy bound,
valid provided that we may use Taylor's expansion of $W$, which here
requires that $h^{1+\delta/2}(\|\nabla^2v_\epsilon\|+\|\nabla^2w_\epsilon\|\big)\to 0$ as $h\to 0$:
\begin{equation*}
\begin{split}
\inf I^h & \leq I^h(u^h) = \int_{\Omega^1}W\Big(\bar q^h q^h \nabla
u^h(A^h)^{-1}(x', hx_3) \Big)\;\mbox{d}(x', x_3) \\ & \leq
C\int_{\Omega^1} \Big(h^{2\delta} \big|\frac{1}{2}(\nabla
 v_\epsilon)^{\otimes 2}+\mbox{sym}\nabla w_\epsilon -S_{2\times 2}  \big|^2 + h^{2+\delta}(\|\nabla^2v_\epsilon\|^2 
+\|\nabla^2w_\epsilon\|^2) \\ & \qquad\qquad \quad + h^{2+\gamma} + h^{4\delta}\Big)\;\mbox{d}(x', x_3)
\end{split}
\end{equation*}
Recalling (\ref{m1}), (\ref{m2}) this leads to:
$$\inf I^h \leq C\big(h^{2\delta}\epsilon^{4a}+h^{2+\delta}\epsilon^{2a-2}+h^{2+\gamma}+h^{4\delta}\big)
= C\big(h^{2\delta+4at}+h^{2+\delta+(2a-2)t}+h^{2+\gamma}+h^{4\delta}\big).$$
Minimizing the right hand side above is equivalent to
maximizing the minimal of the four displayed exponents. We hence
choose $t$ in $\epsilon=h^t$ so that
${2\delta+4at}={2+\delta+(2a-2)t}$, namely $t=\frac{2-\delta}{2a+2}$.
We then obtain:
$$\inf I^h \leq C\big(h^{2\frac{\delta+2a}{a+1}}+h^{2+\gamma}+h^{4\delta}\big)\leq
C\big(h^{(5\delta/3+2/3)-} +h^{2+\gamma}+h^{4\delta}\big) $$
upon recalling that $a\in (0, 1/5)$.  The conclusion of Theorem
\ref{th_general_scaling} follows by a direct inspection.

\smallskip

\subsection{Case $\mathbf{\alpha\geq 4, \gamma\geq 2}$}\label{standard}
Let $\delta\in [2, \gamma]$ and $\alpha\geq 2+\delta$.
Given $v\in \mathcal{C}^\infty(\bar\omega,\R)$, $w\in \mathcal{C}^\infty(\bar\omega,\R^2)$,  and 
$d, \bar d\in \mathcal{C}^\infty(\bar\omega,\R^3)$, we define $u^h\in
\mathcal{C}^\infty(\bar \Omega^h,\R^3)$ by:
\begin{equation}\label{recseq}
u^h = id_3 + h^{\delta/2}\left[\begin{array}{c} 0\\
    v\end{array}\right] + h^{1+\delta/2}\left[\begin{array}{c} w\\ 0\end{array}\right] 
- h^{\delta/2}x_3\left[\begin{array}{c} \nabla v\\ 0\end{array}\right]
+ h^{1+\delta/2}x_3 d + \frac{1}{2}h^{\delta/2}x_3^2\bar d.
\end{equation}
Consider the rotation fields $q^h\in\mathcal{C}^\infty(\bar \omega, SO(3))$:
$$q^h = \exp\Big( h^{\delta/2}\left[\begin{array}{cc} 0 &
    \nabla v\\ -\nabla v & 0\end{array}\right]\Big)= Id_3 + h^{\delta/2}\left[\begin{array}{cc} 0 &
    \nabla v\\ -\nabla v & 0\end{array}\right] - \frac{1}{2}h^\delta \left[\begin{array}{cc} 
    (\nabla v)^{\otimes 2} & 0 \\ 0 & |\nabla v|^2 \end{array}\right]
+ \mathcal{O}(h^{\frac 32 \delta}). $$
Above, the constants on $\mathcal{O}$ depend on $\|\nabla v\|_{L^\infty}$.
Thus, we obtain for all $(x', x_3)\in \Omega^1$:
\begin{equation*}
\begin{split}
q^h\nabla u^h(x', hx_3) = & \;Id_3 + \frac{1}{2}h^\delta \left[\begin{array}{cc} 
    (\nabla v)^{\otimes 2} & 0 \\ 0 & |\nabla v|^2 \end{array}\right]+
h^{1+\delta/2}\left[\begin{array}{cc} \nabla w & 0 \\ 0 & 0\end{array}\right] 
- h^{1+\delta/2}x_3\left[\begin{array}{cc} \nabla^2 v &0\\ 0&0\end{array}\right]
\\ & + h^{1+\delta/2}\left[\begin{array}{cc} 0 &d+x_3\bar d\end{array}\right]
+ \mathcal{O}(h^{2+\delta/2}),
\end{split}
\end{equation*}
where the bound in $\mathcal{O}$ depends on the $L^\infty(\omega)$
norms of: $\nabla v$, $\nabla^2 v$, $\nabla w$, $d$, $\bar d$, $\nabla d$, $\nabla\bar d$.

Recall that: $(A^h(x',hx_3))^{-1} = Id_3 - h^{\alpha/2}S -
h^{1+\gamma/2}x_3B + \mathcal{O}(h^{\alpha\wedge(2+\gamma)})$. Consequently:
\begin{equation*}
\begin{split}
\big(q^h(\nabla u^h) &(A^h)^{-1} \big) (x', hx_3) =  Id_3  - h^{\alpha/2}S - h^{1+\gamma/2}x_3 B
+ \frac{1}{2}h^\delta \left[\begin{array}{cc} 
    (\nabla v)^{\otimes 2} & 0 \\ 0 & |\nabla v|^2 \end{array}\right]
\\ & + h^{1+\delta/2}\left[\begin{array}{cc} \begin{array}{c}\nabla w \\
      0 \end{array} & d\end{array}\right] 
- h^{1+\delta/2}x_3\left[\begin{array}{cc} \begin{array}{c}\nabla^2 v
      \\ 0\end{array} & -\bar d \end{array}\right] + \mathcal{O}(h^{2+\delta/2}).
\end{split}
\end{equation*}
It follows that the integrand $W\big(q^h(\nabla u^h) (A^h)^{-1} \big)$ equals:
\begin{equation*}
\begin{split}
\mathcal{Q}_3\Big(- h^{\alpha/2}S + &\frac{1}{2}h^\delta \left[\begin{array}{cc} 
    (\nabla v)^{\otimes 2} & 0 \\ 0 & |\nabla v|^2 \end{array}\right]
 + h^{1+\delta/2}\left[\begin{array}{cc} \begin{array}{c}\nabla w \\
      0 \end{array} & d\end{array}\right] \\ & \qquad \quad - h^{1+\gamma/2}x_3 B
- h^{1+\delta/2}x_3\left[\begin{array}{cc} \begin{array}{c}\nabla^2 v
      \\ 0\end{array} & -\bar d \end{array}\right] \Big) + o\big(h^{2+\delta}\big),
\end{split}
\end{equation*}
which yields:
\begin{equation*}
\begin{split}
\lim_{h\to 0}&\frac{1}{h^{2+\delta}}I^h(u^h) = 
\frac{1}{24}\lim_{h\to 0} \int_\omega \mathcal{Q}_3\Big(
h^{\gamma/2-\delta/2} B + \left[\begin{array}{cc} \begin{array}{c}\nabla^2 v
      \\ 0\end{array} & -\bar d \end{array}\right] \Big) \dxp \\ & +\frac{1}{2}\lim_{h\to 0}
\int_\omega \mathcal{Q}_3\Big(h^{\alpha/2-(1+\delta/2)}S -\frac{1}{2}h^{\delta/2-1} \left[\begin{array}{cc} 
    (\nabla v)^{\otimes 2} & 0 \\ 0 & |\nabla v|^2 \end{array}\right]
 - \left[\begin{array}{cc} \begin{array}{c}\nabla w \\
      0 \end{array} & d\end{array}\right] \Big)\dxp.  
\end{split}
\end{equation*}
Setting $d $ and $\bar d$ to be affine functions of $\nabla w$,
$(\nabla v)^{\otimes 2}$, $\nabla^2v$ and $S, B$, so that $\mathcal{Q}_3$ above gets replaced by $\mathcal{Q}_2$
evaluated on the principal $2\times 2$ minors of the respective
arguments, we obtain the claimed convergence to the energy functional in (\ref{lowbdbd}):
\begin{equation}\label{upbd}
\lim_{h\to 0}\frac{1}{h^{2+\delta}}I^h(u^h) = I(v,w).
\end{equation}

Finally, we observe that given $v\in W^{2,2}(\omega, \R)$ and $w\in
W^{1,2}(\omega, \R^2)$, one can first find their smooth approximations $v_n$,
$w_n$ and construct the recovery sequence $\{u^h\}_{h\to 0}$ as the
diagonal sequence given by formula in  (\ref{recseq}) with $v$, $w$
replaced by $v_n$, $w_n$. Taking $n=n(h)\to\infty$ as $h\to 0$ at a
sufficiently slow rate, guarantees the same limit as in (\ref{upbd}). 

\subsection{Case $\mathbf{S_{2\times 2}\equiv 0}$ and $\mathbf{\alpha, \gamma\geq 2}$}\label{addi} 
Let $\delta\in [2, \alpha\wedge\gamma]$.
Given $v\in \mathcal{C}^\infty(\bar\omega,\R)$, $w\in \mathcal{C}^\infty(\bar\omega,\R^2)$,  
$d, \bar d\in \mathcal{C}^\infty(\bar\omega,\R^3)$, and denoting
$s=(S_{31}, S_{32})$, we define $u^h\in \mathcal{C}^\infty(\bar \Omega^h,\R^3)$ by:
\begin{equation}\label{recseq5}
\begin{split}
u^h = & \; id_3 + h^{\alpha/2}x_3 \left[\begin{array}{c} 2s  \\ S_{33}\end{array}\right]
\\ & + h^{\delta/2}\left[\begin{array}{c} 0\\
    v\end{array}\right] + h^{1+\delta/2}\left[\begin{array}{c} w\\ 0\end{array}\right] 
- h^{\delta/2}x_3\left[\begin{array}{c} \nabla v\\ 0\end{array}\right]
+ h^{1+\delta/2}x_3 d + \frac{1}{2}h^{\delta/2}x_3^2\bar d.
\end{split}
\end{equation}
As in subsection \ref{standard}, we apply rotations $q^h = \exp\Big( h^{\delta/2}\left[\begin{array}{cc} 0 &
    \nabla v\\ -\nabla v & 0\end{array}\right]\Big)$ and obtain on $\Omega^1$:
\begin{equation*}
\begin{split}
q^h\nabla u^h(x', hx_3) = & \;Id_3 +\frac{1}{2}h^\delta \left[\begin{array}{cc} 
    (\nabla v)^{\otimes 2} & 0 \\ 0 & |\nabla v|^2 \end{array}\right] +
h^{1+\delta/2}\left[\begin{array}{cc} \nabla w & 0 \\ 0 & 0\end{array}\right] 
- h^{1+\delta/2}x_3\left[\begin{array}{cc} \nabla^2 v &0\\ 0&0\end{array}\right]
\vspace{2mm} \\ & +  h^{\alpha/2} \left[\begin{array}{cc} 
0 & \begin{array}{c} 2s\\ S_{33}\end{array}  \end{array}\right]
+ h^{1+\alpha/2}x_3 \left[\begin{array}{cc} 
    2\nabla s & 0 \\ \nabla S_{33} & 0\end{array}\right]
- h^{\delta/2+ \alpha/2} \left[\begin{array}{cc} 
  0 & \begin{array}{c} S_{33}\nabla v  \\ - 2\langle s,
    \nabla v\rangle \end{array}\end{array}\right] \vspace{1mm} \\ &
+ h^{1+\delta/2}\left[\begin{array}{cc} 0 &d+x_3\bar
    d\end{array}\right] + \mathcal{O}(h^{2+\delta/2}),
\end{split}
\end{equation*}
where the bound in $\mathcal{O}$ depends on the $L^\infty(\omega)$
norms of: $\nabla v$, $\nabla^2 v$, $\nabla w$, $d$, $\bar d$, $\nabla d$, $\nabla\bar d$.
We now apply a further rotation $\bar q^h = \exp\Big( h^{\alpha/2}\left[\begin{array}{cc} 0 &
    -s\\ s& 0\end{array}\right]\Big)$ to get:
\begin{equation*}
\begin{split}
\bar q^h &q^h\nabla u^h(x', hx_3) = Id_3 +\frac{1}{2}h^\delta \left[\begin{array}{cc} 
    (\nabla v)^{\otimes 2} & 0 \\ 0 & |\nabla v|^2 \end{array}\right] +
h^{1+\delta/2}\left[\begin{array}{cc} \nabla w & 0 \\ 0 & 0\end{array}\right] 
- h^{1+\delta/2}x_3\left[\begin{array}{cc} \nabla^2 v &0\\ 0&0\end{array}\right]
\vspace{2mm} \\ & +  h^{\alpha/2} S + \frac{1}{2}h^{\alpha/2} 
\left[\begin{array}{cc}  2s^{\otimes 2} & -2S_{33}s \\ 0 & 3|s|^2 \end{array}\right]
+ h^{1+\alpha/2} x_3 \left[\begin{array}{cc}  2\nabla s & 0 \\ \nabla S_{33} & 0\end{array}\right]
+ h^{\delta/2+ \alpha/2} \left[\begin{array}{cc} 
  0 & \begin{array}{c} S_{33}\nabla v  \\ - 2\langle s,
    \nabla v\rangle \end{array}\end{array}\right] \vspace{1mm} \\ &
+ h^{1+\delta/2}\left[\begin{array}{cc} 0 &d+x_3\bar
    d\end{array}\right] + \mathcal{O}(h^{2+\delta/2}),
\end{split}
\end{equation*}

By the second order expansion of the inverse:  $(A^h(x',hx_3))^{-1} = Id_3 - h^{\alpha/2}S -
h^{1+\gamma/2}x_3B + h^\alpha S^2 + \mathcal{O}(h^{1+\delta})$, we obtain:
\begin{equation*}
\begin{split}
\big(\bar q^h q^h(\nabla u^h) &(A^h)^{-1} \big) (x', hx_3) =  Id_3  - h^{1+\gamma/2}x_3 B
+ \frac{1}{2}h^\delta \left[\begin{array}{cc} 
    (\nabla v)^{\otimes 2} & 0 \\ 0 & |\nabla v|^2 \end{array}\right] \vspace{2mm} 
\\ & -\frac{1}{2}h^\alpha \left[\begin{array}{cc}  s^{\otimes 2} &
    2 S_{33}s \\ 0 & -2 S_{33}^2 \end{array}\right]
+ h^{1+\alpha/2}x_3 \left[\begin{array}{cc}  2\nabla s & 0 \\ \nabla S_{33} & 0\end{array}\right]
+ h^{\delta/2+ \alpha/2} \left[\begin{array}{cc} 
  0 & S_{33}\nabla v  \\ 0 & - 2\langle s,
    \nabla v\rangle \end{array}\right] \vspace{2mm} \\ & 
+ h^{1+\delta/2}\left[\begin{array}{cc} \begin{array}{c}\nabla w \\
      0 \end{array} & d\end{array}\right] 
- h^{1+\delta/2}x_3\left[\begin{array}{cc} \begin{array}{c}\nabla^2 v
      \\ 0\end{array} & -\bar d \end{array}\right] + \mathcal{O}(h^{2+\delta/2}).
\end{split}
\end{equation*}
It follows that the integrand $W\big(\bar q^h q^h(\nabla u^h) (A^h)^{-1} \big)$ equals:
\begin{equation*}
\begin{split}
\mathcal{Q}_3\Big( & - \frac{1}{2}h^\alpha \left[\begin{array}{cc}  s^{\otimes 2} &
     S_{33}s \\  S_{33}s & -2 S_{33}^2 \end{array}\right] + h^{\delta/2+ \alpha/2} \left[\begin{array}{cc} 
  0 &\frac{1}{2} S_{33}\nabla v  \\ \frac{1}{2} S_{33}\nabla v  & - 2\langle s,
    \nabla v\rangle \end{array}\right] \\ & + \frac{1}{2}h^\delta \left[\begin{array}{cc} 
    (\nabla v)^{\otimes 2} & 0 \\ 0 & |\nabla v|^2 \end{array}\right]
 + h^{1+\delta/2}\left[\begin{array}{cc} \begin{array}{c}\nabla w \\
      0 \end{array} & d\end{array}\right] \\ & 
+ h^{1+\alpha/2}x_3 \left[\begin{array}{cc}  2\,\mathrm{sym} \nabla s &
    \frac{1}{2}\nabla S_{33} \\ \frac{1}{2}\nabla S_{33} &
    0\end{array}\right]  - h^{1+\gamma/2}x_3 B
+ h^{1+\delta/2}x_3\left[\begin{array}{cc} \begin{array}{c} -\nabla^2 v
      \\ 0 \end{array} & \bar d \end{array}\right] \Big) + o\big(h^{2+\delta}\big),
\end{split}
\end{equation*}
which yields:
\begin{equation*}
\begin{split}
\lim_{h\to 0}&\frac{1}{h^{2+\delta}}I^h(u^h) =  \\ & 
\frac{1}{24}\lim_{h\to 0} \int_\omega \mathcal{Q}_3\Big(h^{\alpha/2-\delta/2}
\left[\begin{array}{cc}  2\,\mathrm{sym} \nabla s & \frac{1}{2}\nabla
    S_{33} \\ \frac{1}{2}\nabla S_{33} & 0\end{array}\right] 
+ h^{\gamma/2-\delta/2} B + \left[\begin{array}{cc} \begin{array}{c} -\nabla^2 v
      \\ 0\end{array} & \bar d \end{array}\right] \Big) \dxp \\ & 
+ \frac{1}{2}\lim_{h\to 0} \int_\omega \mathcal{Q}_3\Big(
- \frac{1}{2} h^{\alpha - (1+\delta/2)} \left[\begin{array}{cc}  s^{\otimes 2} &
     S_{33}s \\  S_{33}s  & -2 S_{33}^2 \end{array}\right] 
+ h^{\alpha/2-1} \left[\begin{array}{cc}  0 & \frac{1}{2} S_{33}\nabla v  \\ \frac{1}{2} S_{33}\nabla v  & - 2\langle s,
    \nabla v\rangle \end{array}\right]  \\ & \qquad\qquad\qquad \quad
+ \frac{1}{2}h^{\delta/2-1} \left[\begin{array}{cc} 
    (\nabla v)^{\otimes 2} & 0 \\ 0 & |\nabla v|^2 \end{array}\right]
 + \left[\begin{array}{cc} \begin{array}{c}\nabla w \\
      0 \end{array} & d\end{array}\right] \Big)\dxp.  
\end{split}
\end{equation*}
In each of the cases of ordering $\alpha,\gamma$ and $\delta$, we may
set $d $ and $\bar d$ to be affine functions of $\nabla w$, 
$(\nabla v)^{\otimes 2}$, $\nabla^2v$ and $S, B$, so that $\mathcal{Q}_3$ above gets replaced by $\mathcal{Q}_2$
evaluated on the principal $2\times 2$ minors of the respective
arguments. Thus, we obtain the claimed convergence to the energy
functional in (\ref{lowbdbd5}) in case when the displacements $v$ and $w$ are smooth.
The general case $v\in W^{2,2}(\omega, \R)$ and $w\in
W^{1,2}(\omega, \R^2)$ follows by a diagonal argument as in subsection
\ref{standard}. 

\section{Discussion of $\inf I$ and the optimality conditions}\label{sec_Helm}

We start by proving two decomposition results for matrices. Recall
that given $F\in \R^{2\times 2}$, its cofactor matrix is: 
$\mathrm{cof}\,F = \left[\begin{array}{cc} F_{22} & - F_{12}\\-F_{21}&F_{11}\end{array}\right]$, while for fields
$w:\omega\to\R^2$ and $\alpha:\omega\to \R$ we denote:
$\mathrm{curl}\,w = \partial_1w_2-\partial_2w_1$ and $\nabla^\perp\alpha = (-\partial_2\alpha, \partial_1\alpha)$.
The congruency symbol in: $a\cong b$ means that $a\leq C b$ and $b\leq Ca$ with a constant
$C$ depending only on $\omega$.

\begin{lemma}\label{lem_dist_hes}
Let $\omega\subset\R^2$ be an open, bounded, simply connected domain with
Lipschitz boundary. For every $F\in L^2(\omega,\mathbb{R}^{2\times 2}_{sym})$
there exist unique  $v\in W^{2,2}(\omega, \R)$, $\phi\in W^{1,2}_0(\omega,\R^2)$ satisfying:
$$F = \nabla^2v + \mathrm{cof}\,\mathrm{sym}\nabla\phi.$$
Moreover, there hold the following equivalences with $\mathrm{curl}\,F$ (taken row-wise):
\begin{equation*}
\begin{split}
\|F-\nabla^2v\|_{L^2}&=\mathrm{dist}_{L^2}\big(F, \{\nabla^2r;~ r\in
W^{2,2}(\omega,\R)\}\big) \\ & \cong \|\mathrm{curl}\,F\|_{H^{-1}(\omega)}=
\mathrm{dist}_{L^2}\big(F, \{\nabla w;~ w\in W^{1,2}(\omega, \R^2)\}\big).
\end{split}
\end{equation*}
\end{lemma}
\begin{proof}
Since the linear space $\{\mathrm{cof}\,\mathrm{sym}\,\nabla
\phi;~\phi\in W_0^{1,2}(\omega,\R^2)\}$ is a closed subspace of
$L^2(\omega,\R^{2\times 2})$, the following minimization problem has the unique solution:
$$\mathrm{minimize}\,\Big\{\int_{\omega} |F-\mathrm{cof}\,\mathrm{sym}\,\nabla
\phi|^2\dxp;~\phi\in W_0^{1,2}(\omega,\R^2)\Big\},$$
identified as the solution to the Euler-Lagrange equation:
\begin{equation}\label{otto}
\int_{\omega}\langle F: \mathrm{cof}\,\mathrm{sym}\,\nabla\alpha\rangle\dxp = 
\int_{\omega}\langle \mathrm{sym}\,\nabla\phi : \mathrm{sym}\,\nabla\alpha\rangle\dxp 
\qquad\mbox{for all }\; \alpha\in W_0^{1,2}(\omega,\R^2).
\end{equation}
By Korn's inequality, the right hand side above is a scalar product on
$W_0^{1,2}(\omega,\R^2)$. A basic application of the Riesz
representation theorem yields existence of the unique solution $\phi$, and:
\begin{equation*}
\|\mathrm{sym}\,\nabla\phi\|_{L^2} = \sup\Big\{\int_{\omega}\langle  F:
\mathrm{cof}\,\mathrm{sym}\,\nabla\alpha\rangle\dxp;~ \alpha\in
W_0^{1,2}(\omega,\R^2),~~ \|\mathrm{sym}\nabla\alpha\|_{L^2}\leq 1\Big\}.
\end{equation*}
Observing that $\langle  F: \mathrm{cof}\,\mathrm{sym}\,\nabla\alpha\rangle
= \langle  F: \mathrm{cof}\,\nabla\alpha\rangle= \big\langle  F:
\left[\begin{array}{c}\nabla^\perp (-\alpha_2)\\
    \nabla^\perp\alpha_1\end{array}\right]\big\rangle$ and using Korn's
inequality again, we get:
\begin{equation*}
\begin{split}
\|\mathrm{sym}\,\nabla\phi\|_{L^2} \cong & \sup\Big\{\int_{\omega}\big\langle  F:
\left[\begin{array}{c}\nabla^\perp \alpha_1\\
    \nabla^\perp\alpha_2\end{array}\right]\big\rangle \dxp;~ \alpha\in
W_0^{1,2}(\omega,\R^2),~~ \|\nabla\alpha\|_{L^2}\leq 1\Big\}\\  \cong &
\sup\Big\{\int_{\omega}\big\langle  (F_{11}, F_{12}), \nabla^\perp
\alpha_1\big \rangle;~ \alpha_1\in W_0^{1,2}(\omega,\R),~~
\|\nabla\alpha_1\|_{L^2}\leq 1\Big\} \\ & + 
\sup\Big\{\int_{\omega}\big\langle  (F_{21}, F_{22}), \nabla^\perp
\alpha_2\big \rangle;~ \alpha_2\in W_0^{1,2}(\omega,\R),~~ \|\nabla\alpha_2\|_{L^2}\leq 1\Big\}
\\ \cong & \, \|\mathrm{curl}\,(F_{11}, F_{12})\|_{H^{-1}(\omega)} + \|\mathrm{curl}\,(F_{21}, F_{22})\|_{H^{-1}(\omega)}
\\ \cong & \, \|\mathrm{curl}\,F\|_{H^{-1}(\omega)} \cong
\mathrm{dist}_{L^2}\big(F, \{\nabla w;~ w\in W^{1,2}(\omega, \R^2)\}\big). 
\end{split}
\end{equation*}

Finally, from (\ref{otto}) we deduce that:
$$\int_{\omega}\big\langle  F - \mathrm{cof}\,\mathrm{sym}\,\nabla\phi :
\left[\begin{array}{c}\nabla^\perp \alpha_1\\
    \nabla^\perp\alpha_2\end{array}\right]\big\rangle \dxp =
0 \qquad\mbox{ for all }\; \alpha_1, \alpha_2\in W_0^{1,2}(\omega,\R).$$
Hence, by de Rham's theorem there must be: $F-\mathrm{cof}\,\mathrm{sym}\,\nabla\phi
=\nabla w$ for some $w\in W^{1,2}(\omega,\R^2)$. Since
$F-\mathrm{cof}\,\mathrm{sym}\,\nabla\phi$ is symmetric, we get that
$\nabla w = \nabla^2 v$ for some $v\in W^{2,2}(\omega, \R)$.
\end{proof}

\begin{lemma}\label{lem_dist_symgrad}
Let $\omega\subset\R^2$ be open, bounded, simply connected, with
Lipschitz boundary. For every $F\in L^2(\omega,\mathbb{R}^{2\times 2}_{sym})$,
there exist unique  $r\in W^{2,2}_0(\omega, \R)$ and $w\in W^{1,2}(\omega,\R^2)$ such that:
$$F = \mathrm{cof}\,\nabla^2r + \mathrm{sym}\nabla w.$$
Moreover, there hold the following equivalences with the scalar field $\mathrm{curl}^T\mathrm{curl}\,F$:
\begin{equation*}
\begin{split}
\|F-\mathrm{sym}\,\nabla w\|_{L^2}&=\mathrm{dist}_{L^2}\big(F,
\{\mathrm{sym}\,\nabla \phi;~ \phi\in
W^{1,2}(\omega,\R^2)\}\big) \cong \|\mathrm{curl}^T\mathrm{curl}\,F\|_{H^{-2}(\omega)}.
\end{split}
\end{equation*}
\end{lemma}
\begin{proof}
Similarly to the proof of Lemma \ref{lem_dist_hes}, we consider the minimization problem:
$$\mathrm{minimize}\,\Big\{\int_{\omega} |F-\mathrm{cof}\,\nabla^2
r|^2\dxp;~ r\in W_0^{2,2}(\omega,\R)\Big\},$$
whose unique solution is given through the orthogonal projection
on the closed subspace $\{\mathrm{cof}\,\nabla^2r;~r\in
W_0^{2,2}(\omega,\R)\}$ of $L^2(\omega,\R^{2\times
  2}_{sym})$. Equivalently, the solution $r$ satisfies:
\begin{equation*}
\int_{\omega}\langle F: \mathrm{cof}\,\nabla^2\alpha\rangle\dxp = 
\int_{\omega}\langle \nabla^2r : \nabla^2\alpha\rangle\dxp 
\qquad\mbox{for all }\; \alpha\in W_0^{2,2}(\omega,\R),
\end{equation*}
and we get:
\begin{equation*}
\|\nabla r\|_{L^2}^2 = \sup\Big\{\int_{\omega}\langle  F: \mathrm{cof}\,\nabla^2\alpha\rangle\dxp;~ \alpha\in
W_0^{2,2}(\omega,\R),~~ \|\nabla^2\alpha\|_{L^2}\leq 1\Big\}
=\|\mathrm{curl}^T\mathrm{curl}\,F\|_{H^{-2}(\omega)}.
\end{equation*}
The last equality above follows by observing that for all $\alpha\in \mathcal{C}^\infty_c(\omega)$ there holds:
\begin{equation}\label{curlo}
\begin{split}
\int_{\omega}\big(\mathrm{curl}^T&\mathrm{curl}\,F\big)\alpha\dxp  =
-\int_\omega \langle \mathrm{curl}\,F, \nabla^\perp\alpha\rangle\dxp
\\ & = \int_\omega \big\langle (F_{11}, F_{12}),\nabla^\perp(-\partial_2\alpha)\big\rangle\dxp + \int_\omega
\big\langle (F_{21}, F_{22}),\nabla^\perp(\partial_1\alpha)\big\rangle\dxp 
\\ & = \int_\omega \big\langle F: \mathrm{cof}\,\nabla^2\alpha\big\rangle\dxp.
\end{split}
\end{equation}

Denoting $\bar F=F - \mathrm{cof}\,\nabla^2r\in L^2(\omega, \R^{2\times 2}_{sym})$, it thus follows that:
$$\int_{\omega}\big\langle  \bar F: \mathrm{cof}\,\nabla^2\alpha\big\rangle \dxp =
0 \qquad\mbox{ for all }\; \alpha\in W_0^{2,2}(\omega,\R).$$
As in (\ref{curlo}), we deduce that: $\mathrm{curl}^T\mathrm{curl}\,\bar F=0$ in distributions.
Hence, for some $v\in L^2(\omega,\R)$:
$$\mathrm{curl}\,\bar F = \nabla v = \mathrm{curl}\,(\mathrm{skew}_2v),\qquad\mbox{where: }\;
\mathrm{skew}_2v=\left[\begin{array}{cc} 0 & v\\-v & 0\end{array}\right].$$
Consequently: $\mathrm{curl}\big(\bar F +\mathrm{skew}_2v\big)=0$, so there
exists $w\in W^{1,2}(\omega,\R^2)$ satisfying: 
$$\bar F+ \mathrm{skew}_2v = \nabla w.$$
Since $\bar F$ is symmetric, this yields: $\bar F=\mathrm{sym}\,\nabla w$. The proof is done.
\end{proof}

\smallskip

\noindent {\bf Proof of Proposition \ref{infbd}}

%
%

{\bf 1.} For $v\in W^{2,2}(\omega,\R)$ we have: $\det\nabla^2v =
\det (\nabla^2v +\bar B) + \det\bar B -\langle \mathrm{cof}\bar B
:\nabla^2v + \bar B\rangle$, so:
\begin{equation*}
\begin{split}
\|\det \nabla^2v + \mathrm{curl}^T\mathrm{curl}\,\bar
S\|_{H^{-2}(\omega)}   \leq \; &
\|\det \bar B +\mathrm{curl}^T\mathrm{curl}\,\bar
S\|_{H^{-2}(\omega)} \\ & + \|\det (\nabla^2v +\bar B)\|_{H^{-2}(\omega)}
+\|\langle \mathrm{cof}\bar B : \nabla^2v + \bar B\rangle\|_{H^{-2}(\omega)}.
\end{split}
\end{equation*}
Observe that:
\begin{equation}\label{grgr}
\begin{split}
& \|\det (\nabla^2v +\bar B)\|_{H^{-2}(\omega)} \leq C
\|\nabla^2v+\bar B\|_{L^2(\omega)}^2, \\ 
& \|\langle \mathrm{cof}\bar B : \nabla^2v + \bar
B\rangle\|_{H^{-2}(\omega)}\leq C\|\bar B\|_{L^\infty(\omega)} \cdot\|\nabla^2v+\bar B\|_{L^2(\omega)},
\end{split}
\end{equation}
which implies:
\begin{equation*}
\begin{split}
\bar I_0(v)\leq C \|\nabla^2v +\bar B\|^2_{L^{2}(\omega)}\Big(1+
\|\bar B\|^2_{L^\infty(\omega)} +
\|\nabla^2v +\bar B\|^2_{L^{2}(\omega)}\Big)+C
\|\det \bar B +\mathrm{curl}^T\mathrm{curl}\,\bar S\|^2_{H^{-2}(\omega)}.
\end{split}
\end{equation*}
The upper bound on $\inf \bar I_0$ follows now by infimizing the
right hand side expression with respect to $v$, and applying Lemma \ref{lem_dist_hes}.

\smallskip

{\bf 2.} For the lower bound, we use (\ref{grgr}) to get:
\begin{equation*}
\begin{split}
\|\det \nabla^2v + \mathrm{curl}^T\mathrm{curl}\,\bar
S\|_{H^{-2}(\omega)}^2   \geq \; & \frac{1}{2}
\|\det \bar B +\mathrm{curl}^T\mathrm{curl}\,\bar
S\|_{H^{-2}(\omega)} \\ & -c \|\nabla^2v +\bar B\|^2_{L^{2}(\omega)}
\big(\|\bar B\|^2_{L^\infty(\omega)} + \|\nabla^2v +\bar B\|^2_{L^{2}(\omega)}\big).
\end{split}
\end{equation*}
The established upper bound yields along a minimizing sequence
$\{v_n\}_{n\to\infty}$ of $\bar I_0$:
\begin{equation*}
\begin{split}
\|\nabla^2v_{n} +\bar B\|^2_{L^{2}(\omega)} \leq & \; C
\|\mathrm{curl}\,\bar B\|_{H^{-1}(\omega)}^2\big(1+\|\bar
B\|^2_{L^\infty(\omega)} + \|\mathrm{curl}\,\bar
B\|_{H^{-1}(\omega)}^2\big) \\ & + C
\|\det \bar B +\mathrm{curl}^T\mathrm{curl}\,\bar S\|^2_{H^{-2}(\omega)},
\end{split}
\end{equation*}
so consequently, for every $\epsilon\leq 1$ there holds:
\begin{equation*}
\begin{split}
\inf \bar I_0 = & \; \lim_{n\to\infty} \bar I_0(v_n) \\ \geq & \; \|\nabla^2v_n +\bar B\|^2_{L^{2}(\omega)}
\Big(1 - \epsilon\big( \|\bar B\|^2_{L^\infty(\omega)} +
\|\det \bar B +\mathrm{curl}^T\mathrm{curl}\,\bar S\|^2_{H^{-2}(\omega)}
\\ & \qquad\qquad\qquad \qquad\qquad \qquad 
+ \|\mathrm{curl}\,\bar B\|_{H^{-1}(\omega)}^2\big(1+\|\bar
B\|^2_{L^\infty(\omega)} + \|\mathrm{curl}\,\bar
B\|_{H^{-1}(\omega)}^2\big)\big)\Big) \\ & + c\epsilon  
\|\det \bar B +\mathrm{curl}^T\mathrm{curl}\,\bar S\|^2_{H^{-2}(\omega)},
\end{split}
\end{equation*}
The result now follows by invoking  Lemma \ref{lem_dist_hes} and
taking $2\epsilon$ to be the minimum of $1$ and the inverse of the expression:
$1/a = \|\bar B\|^2_{L^\infty(\omega)} +
\|\det \bar B +\mathrm{curl}^T\mathrm{curl}\,\bar S\|^2_{H^{-2}(\omega)}
+ \|\mathrm{curl}\,\bar B\|_{H^{-1}(\omega)}^2\big(1+\|\bar
B\|^2_{L^\infty(\omega)} + \|\mathrm{curl}\,\bar
B\|_{H^{-1}(\omega)}^2\big)$. This completes the proof.
\endproof

\section{Connection to curvature}\label{sec_curv}

In this section, we compute the Riemann curvatures in (\ref{curv_1}) and (\ref{curv_2}).
Recall that the Riemannian metrics on $\Omega^h$, induced by the
prestrain tensors $A^h$, are: 
$$\G^h= (A^h)^TA^h = Id_3 + 2h^{\alpha/2}S + h^\alpha S^2 +
x_3\big(2h^{\gamma/2}B + 2 h^{(\alpha+\gamma)/2}\mbox{sym}(SB)\big) +
x_3^2h^\gamma B^2.$$
The Christoffel symbols of $\mathcal{G}^h$ are gathered in matrices  $\{\Gamma_a\}_{a=1\ldots 3}$, where:
$$\Gamma_a = [\Gamma_{ac}^b]_{b,c=1\ldots 3}, \qquad 
\Gamma_{ac}^b=\frac{1}{2}\G^{bm}\big(\partial_a\G_{mc} + \partial_c\G_{ma}
-\partial_m\G_{ac}\big).$$
Writing the Riemann curvatures as $\{R^\cdot_{\cdot,
  cd}\}_{c,d=1\ldots 3}$ and $\{R_{\cdot\cdot, cd}\}_{c,d=1\ldots 3}$, so that $R^\cdot_{\cdot,
  cd}=[R^a_{b,cd}]_{a,b=1\ldots 3}$ and $R_{\cdot \cdot,
  cd}=[R_{ab,cd}]_{a,b=1\ldots 3}$, recall that:
$$R^\cdot_{\cdot, cd}= \big(\partial_c\Gamma_d - \partial_d\Gamma_c\big) +
\big(\Gamma_c\Gamma_d-\Gamma_d\Gamma_c\big), \qquad R_{\cdot\cdot, cd}= \mathcal{G}^hR^\cdot_{\cdot,cd}$$
We now compute $R_{\cdot\cdot, cd}$ at $x_3=0$, in the two cases
studied in this paper.

\subsection{Case $\mathbf{\alpha\geq 4}$, $\mathbf{\gamma\geq 2}$}\label{sec_uno}

We directly compute for all $a,b,c=1\ldots 2$:
\begin{equation*}
\begin{split}
& \Gamma_{ac}^b = h^{\alpha/2}\big(\partial_aS_{bc}+\partial_cS_{ab}-\partial_bS_{ac}\big)
+ h^{\gamma/2}x_3\big(\partial_aB_{bc}+\partial_cB_{ab}-\partial_bB_{ac}\big) 
+ 2h^\gamma x_3 B_{b3}B_{ac}  +e, \\
& \Gamma_{ac}^3 = h^{\alpha/2}\big(\partial_aS_{c3}+\partial_cS_{a3}\big) -h^{\gamma/2}B_{ac}
+ h^{\gamma/2}x_3\big(\partial_aB_{c3}+\partial_cB_{a3}\big) 
\\ & \qquad\quad - h^\gamma x_3 \big((B^2)_{ac} - 2B_{33}B_{ac}\big) + e,\\
& \Gamma_{a3}^b = h^{\alpha/2}\big(\partial_aS_{3b}-\partial_bS_{a3}\big) +h^{\gamma/2}B_{ab}
+ h^{\gamma/2}x_3\big(\partial_aB_{b3} - \partial_bB_{a3}\big) 
\\ & \qquad\quad - h^\gamma x_3 \big((B^2)_{ab} - 2B_{b3}B_{a3}\big) + e,\\
& \Gamma_{a3}^3 = h^{\alpha/2}\partial_aS_{33} + h^{\gamma/2}\partial_aB_{33}
- h^{\alpha}x_3\big((B^2)_{a3} - B_{33}B_{a3}\big) + e,\\
& \Gamma_{33}^b = 2h^{\gamma/2}B_{b3} -h^{\alpha/2}\partial_bS_{33}
- h^{\gamma/2}x_3\partial_bB_{33}- 2h^\gamma x_3 \big((B^2)_{b3}  - B_{33}B_{b3}\big) + e, \\
& \Gamma_{33}^3 = h^{\gamma/2}B_{33}
- h^\gamma x_3 \big(3(B^2)_{33} - 2B_{33}^2\big) + e,
\end{split}
\end{equation*}
where $e$ denotes the error terms of order $\mathcal{O}\big(h^{\alpha\wedge (\alpha+\gamma)/2}+
h^{(\alpha+\gamma)/2} x_3+ h^\gamma x_3^2\big)$.

\smallskip

\textbullet $~$ Consequently, we obtain the following skew-symmetric
matrix fields at $x_3=0$, whose lowest order terms are:
\begin{equation*}
\begin{split}
& \partial_1\Gamma_2-\partial_2\Gamma_1 \simeq \left[\begin{array}{ccc}
0 & -h^{\gamma/2} \mbox{curl}^T\mbox{curl} \, S_{2\times 2} &
-h^{\alpha/2}\partial_1\mbox{curl} (S_{13}, S_{23}) + h^{\gamma/2}(\mbox{curl} B_{2\times 2})_1\\
\cdot & 0 & -h^{\alpha/2}\partial_2\mbox{curl} (S_{13}, S_{23}) + h^{\gamma/2}(\mbox{curl} B_{2\times 2})_2\\
\cdot & \cdot & 0 \end{array}\right],\\
&\Gamma_1\Gamma_2 - \Gamma_2\Gamma_1\simeq \left[\begin{array}{ccc}
0 & -h^\gamma\det B_{2\times 2} & 0\\
\cdot & 0 & 0\\ \cdot & \cdot & 0 \end{array}\right].
\end{split}
\end{equation*}
Thus, the lowest order terms of the curvatures computed at $x_3=0$ are:
\begin{equation*}
R^\cdot_{\cdot, 12}\simeq\left[\begin{array}{cc} \left[\begin{array}{cc} 0 & -
    h^{\alpha/2} \mbox{curl}^T\mbox{curl}\,S_{2\times 2} -h^\gamma \det B_{2\times 2}\big) \\
\cdot & 0\end{array}\right] &
- h^{\alpha/2}\nabla '\mbox{curl}(S_{13}, S_{23}) + h^{\gamma/2}\mbox{curl}B_{2\times 2} \\
\cdot & 0 \end{array}\right],
\end{equation*}
where we denoted: $\nabla '=(\partial_1, \partial_2)$. 

\smallskip

\textbullet $~$ For every $a=1\ldots 2$ we have, at $x_3=0$, with the same
notation as above: 
\begin{equation*}
\begin{split}
& \big(\partial_a\Gamma_3-\partial_3\Gamma_a\big)_{2\times 2} \simeq -2h^\gamma\left[\begin{array}{cc}
B_{13}B_{a1} & B_{13}B_{a2}\\ B_{23}B_{a1} & B_{23}B_{a2} \end{array}\right] +
\left[\begin{array}{cc} 0 & -h^{\alpha/2}\partial_a\mbox{curl}(S_{13}, S_{23})
+ h^{\gamma/2}\big(\mbox{curl} B_{2\times 2} \big)_a\\
\cdot & 0 \end{array}\right], \\ &
\big(\big(\partial_a\Gamma_3-\partial_3\Gamma_a\big)_{13},
\big(\partial_a\Gamma_3-\partial_3\Gamma_a\big)_{23} \big)  \simeq h^\gamma
\big(((B^2)_{a1},(B^2)_{a2})-2B_{a3}(B_{13}, B_{23})\big), \\ &
\qquad\qquad \qquad\qquad \qquad\qquad \qquad\qquad \quad -
h^{\alpha/2}\partial_a\nabla ' S_{33} + h^{\gamma/2}\big(\nabla ' B_{a3}+\partial_a(B_{13}, B_{23})\big),\\
& \big(\big(\partial_a\Gamma_3-\partial_3\Gamma_a\big)_{31},
\big(\partial_a\Gamma_3-\partial_3\Gamma_a\big)_{32} \big) \simeq h^\gamma
\big(((B^2)_{a1},(B^2)_{a2})-2B_{33}(B_{a1}, B_{a2})\big), \\ &
\qquad\qquad \qquad\qquad \qquad\qquad \qquad\qquad \quad +
h^{\alpha/2}\partial_a\nabla ' S_{33} - h^{\gamma/2}\big(\nabla '
B_{a3}+\partial_a(B_{13}, B_{23})\big)\\
& \big(\partial_a\Gamma_3-\partial_3\Gamma_a\big)_{33}
\simeq 2h^\gamma \big(B_{13}B_{a1} + B_{23}B_{a2}\big).
\end{split}
\end{equation*}
We further have:
\begin{equation*}
\begin{split}
&\big(\Gamma_a\Gamma_3 - \Gamma_3\Gamma_a\big)_{2\times 2}\simeq 
2h^\gamma\left[\begin{array}{cc} B_{13}B_{a1} & B_{13}B_{a2}\\ B_{23}B_{a1} &
B_{23}B_{a2} \end{array}\right], \\ &
\big(\big(\Gamma_a\Gamma_3-\Gamma_3\Gamma_a\big)_{13},
\big(\Gamma_a\Gamma_3-\Gamma_3\Gamma_a\big)_{23} \big) =
\big(\big(\Gamma_a\Gamma_3-\Gamma_3\Gamma_a\big)_{31},
\big(\Gamma_a\Gamma_3-\Gamma_3\Gamma_a\big)_{32} \big) \\ &
\qquad\qquad \qquad\qquad \qquad \qquad\qquad \qquad\quad 
\simeq h^\gamma\big(B_{33}(B_{a1}, B_{a2}) - B_{2\times 2} (B_{a1}, B_{a2})\big)\\ &
\big(\Gamma_a\Gamma_3-\Gamma_3\Gamma_a\big)_{33}
\simeq - 2h^\gamma \big(B_{13}B_{a1} + B_{23}B_{a2}\big).
\end{split}
\end{equation*}
We thus obtain the corresponding lowest order terms of curvatures, computed at $x_3=0$:
\begin{equation*}
\begin{split}
& R^\cdot_{\cdot, a3}\simeq \left[\begin{array}{cc} \left[\begin{array}{cc} 0 & 
    - h^{\alpha/2} \partial_a \mbox{curl}\,S_{2\times 2} +
h^{\gamma/2}\big(\mbox{curl} B_{2\times 2}\big)_a \\
\cdot & 0\end{array}\right] & v\\ \cdot & 0 \end{array}\right], \; \mbox{ where: }\;\\
& v = - h^{\alpha/2}\nabla '\partial_a S_{33} + h^{\gamma/2}\big(\nabla '
B_{a3} + \partial_a (B_{13}, B_{23})\big).
\end{split}
\end{equation*}
Finally, we observe that $R_{ab, cd}\simeq R^a_{b,cd}$ at $x_3=0$.

\subsection{Case $\mathbf{S_{2\times 2}\equiv 0}$ and $\mathbf{\alpha, \gamma\geq 2}$}
\label{sec_due}

We compute for all $a,b,c=1\ldots 2$:
\begin{equation*}
\begin{split}
& \Gamma_{ac}^b = h^{\alpha}\Big(\frac{1}{2} S_{3c}(\partial_aS_{3b}-\partial_bS_{3a}) 
+ \frac{1}{2} S_{3a}(\partial_cS_{3b}-\partial_bS_{3c}) - \frac{3}{2} S_{3b}(\partial_aS_{3c}+\partial_cS_{3a}) \Big)
+ 2h^{(\alpha+\gamma)/2} S_{b3}B_{ac} \\ & \qquad\quad
+ h^{\gamma/2} x_3\big(\partial_aB_{bc}+\partial_cB_{ab}-\partial_bB_{ac}\big)
+ 2 h^\gamma x_3 B_{b3}B_{ac} \\ & \qquad\quad
+ h^{(\alpha+\gamma)/2}x_3\Big(\partial_a\mbox{sym}(SB)_{bc}+\partial_c\mbox{sym}(SB)_{ab}
-\partial_b\mbox{sym}(SB)_{ac} \\ & \qquad\qquad \qquad\qquad \qquad - 2S_{b3}(\partial_aB_{3c}+\partial_cB_{a3}) - 
2B_{b3}(\partial_aS_{3c}+\partial_cS_{a3})\Big)  +e, \\
& \Gamma_{ac}^3 = h^{\alpha/2}\big(\partial_aS_{c3}+\partial_cS_{a3}\big) 
+ h^{\alpha}\Big(\frac{1}{2}S_{3c}\partial_aS_{33}+S_{3a}\partial_cS_{33}
-\frac{3}{2}S_{33}(\partial_aS_{3c}+\partial_cS_{3a})\Big) \\ & \qquad\quad
- h^{\gamma/2}B_{ac} + h^{(\alpha+\gamma)/2}\big(-\mbox{sym}(SB)_{ac}+B_{33}B_{ac}\big) \\ & \qquad\quad
+ h^{\gamma/2}x_3\big(\partial+aB_{3c}+\partial_cB_{3a}\big) +h^\gamma
x_3\big(-(B^2)_{ac}+B_{33}B_{ac}\big)
\\ & \qquad\quad +  h^{(\alpha+\gamma)/2}x_3\Big(\partial_a\mbox{sym}(SB)_{3c}+\partial_c\mbox{sym}(SB)_{3a}
- 2(S\partial_aB)_{3c}-2(A\partial_cB)_{3a} \\ & \qquad\qquad \qquad\qquad \qquad +2\langle (S_{31},
S_{32}),\nabla'B_{ac}\rangle - 2B_{33}(\partial_aS_{3c}+ \partial_cS_{3a})\Big)+e,\\
& \Gamma_{a3}^b = h^{\alpha/2}\big(\partial_aS_{3b}-\partial_bS_{a3}\big) +h^{\gamma/2}B_{ab}
\\ & \qquad\quad + h^\alpha \Big( \frac{1}{2}S_{33}(\partial_aS_{b3}-\partial_bS_{a3}) -
\frac{3}{2}S_{b3}\partial_aS_{33}-\frac{1}{2}S_{a3}\partial_bS_{33}\Big)
+ h^{(\alpha+\gamma)/2}\mbox{sym}(SB)_{ab} \\ & \qquad\quad
+ h^{\gamma/2}x_3\big(\partial_aB_{b3} - \partial_bB_{a3}\big) 
- h^\gamma x_3 \big((B^2)_{ab} - 2B_{b3}B_{a3}\big) \\ & \qquad\quad + 
h^{(\alpha+\gamma)/2}x_3\Big(\partial_a\mbox{sym}(SB)_{b3}-\partial_b\mbox{sym}(SB)_{a3} 
\\ & \qquad\qquad \qquad\qquad \qquad  - 2(S\partial_aB)_{b3}-2B_{b3}\partial_aS_{33}
+2\langle (B_{b1}, B_{b2}),\nabla' S_{a3}\rangle\Big) +e, \\
& \Gamma_{a3}^3 = h^{\alpha/2}\partial_aS_{33} + h^\alpha \Big(
\frac{1}{2}((\partial_aS)S)_{33}-\frac{3}{2}(S\partial_aS)_{33} +2 \langle (S_{31}, S_{32}), \partial' S_{a3}\rangle\Big)
\\ & \qquad\quad
+ 2h^{(\alpha+\gamma)/2}\big(-(SB)_{3a}+S_{33}B_{3a}\big) \\ & \qquad\quad 
+ h^{\gamma/2}x_3\partial_aB_{33}+2h^{\gamma}x_3 \big(-(B^2)_{3a}+B_{33}B_{3a}\big)
\\ & \qquad\quad + h^{(\alpha+\gamma)/2}x_3\Big(\partial_a\mbox{sym}(SB)_{33}
- 2(S\partial_aB)_{33} \\ & \qquad\qquad \qquad\qquad \qquad  - 2(B\partial_aS)_{33} + 2 \langle (S_{31},
S_{32}), \nabla' B_{a3}\rangle + \langle (B_{31}, B_{32}),\nabla' S_{a3}\rangle\Big) + e,
\end{split}
\end{equation*}
\begin{equation*}
\begin{split}
& \Gamma_{33}^b = 2h^{\gamma/2}B_{b3} -h^{\alpha/2}\partial_bS_{33}
-\frac{1}{2}h^\alpha \partial_b(S^2)_{33} + 2h^{(\alpha+\gamma)/2}\big(\mathrm{sym}(SB)_{b3}-S_{b3}B_{33}\big)
\\ & \qquad\quad - h^{\gamma/2}x_3\partial_bB_{33}- 2h^\gamma x_3 \big((B^2)_{b3}  - B_{33}B_{b3}\big) 
\\ & \qquad\quad + h^{(\alpha+\gamma)/2}x_3\Big(-\partial_b\mbox{sym}(SB)_{33}
+ \langle (B_{31}, B_{32}),\nabla' S_{33}\rangle\Big) + e, \\
& \Gamma_{33}^3 = h^{\gamma/2}B_{33} + 2h^\alpha\langle (S_{31},S_{32}),\nabla' S_{33}\rangle 
+ h^{(\alpha+\gamma)/2}\Big(\mbox{sym}(SB)_{33} - 4(SB)_{33} + 2S_{33}B_{33}\Big) 
\\ & \qquad\quad + h^\gamma x_3 \big(-3(B^2)_{33} - 2B_{33}^2\big) 
\\ & \qquad\quad + 2h^{(\alpha+\gamma)/2}x_3\Big(\langle (S_{31},
S_{32}), \nabla' B_{33}\rangle + \langle (B_{31}, B_{32}),\nabla' S_{33}\rangle\Big) + e,
\end{split}
\end{equation*}
where $e$ denotes the error terms of order:
$\mathcal{O}\big(h^{\alpha+(\alpha\wedge \gamma)/2}+
h^{(\alpha+\gamma)/2 + (\alpha\wedge\gamma)/2} x_3+ h^\gamma x_3^2\big)$.

\smallskip

\textbullet $~$ Proceeding as in subsection \ref{sec_uno}, we obtain
the following expressions for the lowest order terms of tangental
curvatures at $x_3=0$, where we denote $s=(S_{31}, S_{32})$:
\begin{equation*}
\begin{split}
& R_{12,12} \simeq  \; h^\alpha \big(-3 \det\nabla' s
+ \big\langle\nabla'\,\mbox{curl}\, s, s^\perp\big\rangle\big) 
- h^\gamma\det B_{2\times 2} + 2h^{(\alpha+\gamma)/2} \big\langle
B_{2\times 2}: \mbox{cof}\,\nabla' s\big\rangle,\\ &
(R_{13,12}, R_{23,12})  \simeq  -h^{\alpha/2}\nabla' \mbox{curl}\, s+ h^{\gamma/2} \mbox{curl}B_{2\times 2}.
\end{split}
\end{equation*}
It is also instructive to directly check that:
\begin{equation*}
\begin{split}
4\det\big(\mathrm{sym}\,\nabla' s\big) +
\frac{1}{2}\mathrm{curl}^T\mathrm{curl}(S_{31}, S_{32})^{\otimes
  2} = 3\det\nabla' s-\big\langle \nabla'\mathrm{curl}\,s, s^\perp\big\rangle,
\end{split}
\end{equation*}
which justifies the equivalence of the two conditions in
Theorem \ref{th_main150} (i) when $\alpha=\gamma=2$.

\smallskip

\textbullet $~$ The lowest order terms of the remaining curvatures are contained in the following skew-symmetric
matrix field at $x_3=0$, where $a=1\ldots 2$:
\begin{equation*}
\begin{split}
& R_{\cdot\cdot, a3}\simeq \left[\begin{array}{cc} \left[\begin{array}{cc} 0 & 
    - h^{\alpha/2} \partial_a \mbox{curl}(S_{31}, S_{32}) +
h^{\gamma/2}\big(\mbox{curl} B_{2\times 2}\big)_a \\
\cdot & 0\end{array}\right] & v\\ \cdot & 0 \end{array}\right], \; \mbox{ where: }\;\\
& v = - h^{\alpha/2}\nabla '\partial_a S_{33} + h^{\gamma/2}\big(\nabla '
B_{a3} + \partial_a (B_{13}, B_{23})\big).
\end{split}
\end{equation*}
We observe that the non-tangential curvatures above are the same as in
subsection \ref{sec_uno}.


\begin{thebibliography}{11}

\bibitem{BS} {\sc de Benito Delgado, M. and Schmidt, B.}, {\em A
    hierarchy of multilayered plate models}, preprint.

\bibitem{CS} {\sc Cao, W. and Szekelyhidi, L.},
{\em Very weak solutions to the two-dimensional Monge-Ampere
  equation}, Science China Mathematics, {\bf 62}, pp. 1041-1056, (2019).

\bibitem{CDS} {\sc Conti, S., De Lellis C. and Szekelyhidi, L.}, {\em $h$-principle and rigidity for $C
^{1,\alpha}$ isometric embeddings}, Proceedings of the Abel Symposium (2010).

\bibitem{15} {\sc Dias, M., Hanna J. and Santangelo, C.}, {\em
    Programmed buckling by controlled lateral swelling in a thin
    elastic sheet}, Phys. Rev. E 84, (2011), 036603. 

\bibitem{FJM} {\sc Friesecke, G., James, R.  M\"uller, S.}, {\em A hierarchy 
of plate models derived from nonlinear elasticity by gamma-convergence}, 
Arch. Ration. Mech. Anal.,  {\bf 180(2)}, pp. 183-236, (2006).

\bibitem{22} {\sc Gladman, A., Matsumoto, E., Nuzzo, R., Mahadevan,
    L. and Lewis, J.}, {\em Biomimetic 4D printing}, Nature Materials,
  {\bf 15}, pp. 413–418, (2016).

\bibitem{JZ} {\sc Jimenez-Bolanos, S. and Zemlyanova, A.},
{\em Relative bending energy for weakly prestrained shells},
Rocky Mountain J. Math., {\bf 50(3)}, pp.1001-1020,  (2020).

\bibitem{J} {\sc Jimenez-Bolanos, S.},
{\em Thin films prestrained by shallow curvature lead to the
  Monge-Ampere constraints}, in preparation.

\bibitem{25} {\sc Kempaiah, R. and Nie, Z.}, {\em From nature to
    synthetic systems: shape transformation in soft materials},
  J. Mater. Chem. B, {\bf 2}, pp. 2357–2368, (2014).

\bibitem{26} {\sc Kim, J., Hanna, J., Byun, M., Santangelo, C. and
    Hayward, R.}, {\em Designing responsive buckled surfaces by
    halftone gel lithography}, Science, {\bf 335}, pp. 1201–1205, (2012). 

\bibitem{27} {\sc Klein, Y.,  Efrati, E. and Sharon, E.}, {\em Shaping
    of elastic sheets by prescription of non-Euclidean metrics},
  Science, {\bf 315}, pp. 1116–1120, (2007).

\bibitem{Lew_last} {\sc Lewicka, M.},
{\em Quantitative immersability of Riemann metrics and the infinite
  hierarchy of prestrained shell models}, Arch. Ration. Mech. Anal.,
{\bf 236}, pp. 1677--1707, (2020).

\bibitem{lemapa1} {\sc Lewicka, M., Mahadevan, L. and Pakzad, R.},
{\em The Foppl-von Karman equations for plates with incompatible strains},
Proceedings of the Royal Society A, {\bf 467}, pp. 402-426, (2011).

\bibitem{lemapa2} {\sc Lewicka, M., Mahadevan, L.,  and Pakzad, R.},
{\em Models for elastic shells with incompatible strains}, Proceedings
of the Royal Society A, {\bf 470} 20130604 (2014).

\bibitem{LMP} {\sc Lewicka, M., Mora, M. and Pakzad, R.},
{\em A nonlinear theory for shells with slowly varying thickness},
Comptes Rendus Mathematique, {\bf 347(3-4)}, pp. 211-216,  (2009). 

\bibitem{LOP} {\sc Lewicka, M., Ochoa, P. and Pakzad, R.},
{\em Variational models for prestrained plates with Monge-Ampere constraint},
Differential and Integral Equations, {\bf 29(9-10)}, pp. 861-898,  (2015). 

\bibitem{LP} {\sc Lewicka, M. and Pakzad, R.},
{\em Convex integration for the Monge-Ampere equation in two
  dimensions}, Analysis and PDE, {\bf 10(3)}, pp. 695-727, (2017).

\bibitem{LRR} {\sc Lewicka, M., Raoult, A. and Ricciotti, D.},
{\em Plates with incompatible prestrain of higher order}, Annales de
l'Institut Henri Poincare (C) Non Linear Analysis, {\bf 34(7)}, pp. 1883--1912, (2017).

\bibitem{51} {\sc Wei, Z., Jia, J., Athas, J., Wang, C., Raghavan, S.,
   Li, T. and Nie, Z.}, {\em Hybrid hydrogel sheets that undergo pre-programmed shape
transformations}, Soft Matter, {\bf 10}, pp. 8157–8162, (2014).

\end{thebibliography}
\end{document}